\documentclass[12pt]{elsarticle}

\usepackage{amssymb,amsthm, url,amsmath,graphicx}
\usepackage{tikz}
\usepackage{pgfplots}
\usepackage{pgfplotstable}

\usepackage{listings}
\definecolor{DarkBlue}{rgb}{0.00,0.00,0.55}
\definecolor{DarkRed}{rgb}{0.55,0.00,0.00}
\definecolor{DarkGreen}{rgb}{0.00,0.55,0.00}
\definecolor{Bittersweet}{rgb}{1.0, 0.44, 0.37}
\definecolor{Purple}{rgb}{0.5, 0.0, 0.5}

\usepackage{cleveref}

\usepackage{subcaption}

\newtheorem{theorem}{Theorem}[section]
\newtheorem{proposition}{Proposition}[section]
\newtheorem{corollary}{Corollary}[section]
\def\bfx{{\mathbf{x}}}
\def\bfu{{\mathbf{u}}}
\def\bfv{{\mathbf{v}}}

\def\hdiv{{H(\mathrm{div})}}
\def\hcurl{{H(\mathrm{curl})}}

\makeatletter
\newcommand{\opnorm}{\@ifstar\@opnorms\@opnorm}
\newcommand{\@opnorms}[1]{%
        \left|\mkern-1.5mu\left|\mkern-1.5mu\left|
        #1
        \right|\mkern-1.5mu\right|\mkern-1.5mu\right|
}
\newcommand{\@opnorm}[2][]{%
        \mathopen{#1|\mkern-1.5mu#1|\mkern-1.5mu#1|}
        #2
        \mathclose{#1|\mkern-1.5mu#1|\mkern-1.5mu#1|}
}
\makeatother

\pgfplotsset{compat=1.13}

\begin{document}

\begin{frontmatter}


\title{Preconditioning mixed finite elements for tide models}

\tnotetext[t1]{This work was supported by NSF 1912653.}
\author[bu]{Robert C. Kirby\corref{cor1}}
\ead{robert\_kirby@baylor.edu}

\author[bu]{Tate Kernell}
\ead{otkernell@gmail.com}

\cortext[cor1]{Corresponding author}
\address[bu]{Department of Mathematics, Baylor University, One Bear Place \#97328, Waco, TX 76798-7328, United States}


\begin{abstract}
We describe a fully discrete mixed finite element method for the linearized rotating shallow water model, possibly with damping.  While Crank-Nicolson time-stepping conserves energy in the absence of drag or forcing terms and is not subject to a CFL-like stability condition, it requires the inversion of a linear system at each step.  We develop weighted-norm preconditioners for this algebraic system that are nearly robust with respect to the physical and discretization parameters in the system.  Numerical experiments using Firedrake support the theoretical results.
\end{abstract}


\end{frontmatter}


\section{Introduction}
Accurate modeling of tides plays an important role in several disciplines.  For example, geologists use tide models to help understand sediment transport and coastal flooding, while oceanographers study tides to discern mechanisms sustaining global circulation~\cite{GaKu2007,MuWu1998}.  Finite element methods making use of unstructured (typically triangular) meshes are attractive to handle irregular coastlines and topography~\cite{We_etal2010}.  
In many situations, it is sufficient to use a linearized shallow water
model with rotation and a parameterized drag term.  In particular, the
literature contains many
papers~\cite{CoLaReLe2010,CoHa2011,Ro2005,RoRoPo2007,RoBe2009,RoRo2008}
studying mixed finite element pairs as horizontal discretizations for
ocean and atmosphere models, and we continue study of this case here.

Much of the literature relates to dispersion relations and enforcement of conservation principles by mixed methods, our prior work in this area has been to focus on energy estimates.  In~\cite{CoKi}, we gave a careful account of the effect of linear bottom friction in semidiscrete mixed methods, showing that, absent forcing, one obtained exponential damping of a natural energy functional.  This allowed estimates of long-time stability and optimal-order \emph{a priori} error estimates.  Then, we handled the (much more delicate) case of a broad family of nonlinear damping terms in~\cite{cotter2018mixed}.  In this case, the energy decay is sub-exponential (typically bounded by a power law) but still strong enough to admit long-time stability and error estimates.

While our work in~\cite{CoKi,cotter2018mixed} focused on the
semidiscrete mixed finite element case, we now turn to certain issues
related to time-stepping.  Crank-Nicolson time-stepping is
second-order accurate, A-stable (not subject to  CFL-like stability
condition), and exactly energy
conserving in the absence of forcing and damping.  However, because it is implicit, it requires the solution of a system of algebraic equations at each time step.  For linear damping models, this system is linear, but nonlinear otherwise.  The point of this paper is to develop robust preconditioners for the linear system (or Jacobian of the nonlinear one) for use in conjunction with a Krylov method such as GMRES~\cite{saad1986gmres}.

In addition to the mesh size and time step, our model also depends on
a number of physical parameters, described in the following section.
Our goal is to design a preconditioner that enables GMRES to converge 
with an overall iteration counts that depend as
little as possible on these parameters.  We follow the technique of
using weighted-norm preconditioners~\cite{Arnold1997}.  Here, one designs an inner product with respect to which the variational problem is bounded with bounded inverse, and such bounds should depend weakly, if at all, on parameters.

In the rest of the paper, we describe the particular tide model of interest and its discretization in Section~\ref{sec:model}.  This includes Crank-Nicolson time-stepping and a comparison to a symplectic Euler method.  Then, we turn to preconditioning the Crank-Nicolson system in Section~\ref{sec:precond}.  After analyzing a simple block-diagonal preconditioner with scaled mass matrices, we develope and analyze a parameter-weighted inner product on $\hdiv \times L^2$.  Our estimate shows that the preconditioned system has an intrinsic time scale determined by the Rossby number that must be resolved by the time step.  This does not seem to be a major practical constraint.  After discussion and analysis of these preconditioners, we turn to numerical experiments validating the theory in Section~\ref{sec:numres} and draw some conclusions in Section~\ref{sec:conc}.

\section{Description of finite element tidal model}
\label{sec:model}
The nondimensional linearized rotating shallow water
model with linear drag and forcing on a two
dimensional surface $\Omega$ are given by
\begin{equation}
  \label{eq:theinitialpde}
\begin{split}
u_t + \frac{f}{\epsilon} u^\perp + \frac{\beta}{\epsilon^2} \nabla
\left( \eta - \eta^\prime \right) + C u & = 0, \\
\eta_t + \nabla \cdot \left( H u \right) & = 0,
\end{split}
\end{equation}
where $u$ is the nondimensional two dimensional velocity field tangent
to $\Omega$, $u^\perp=(-u_2,u_1)$ is the velocity rotated by $\pi/2$, $\eta$ is the nondimensional free surface elevation above
the height at state of rest, $\nabla\eta'$ is the (spatially varying)
tidal forcing, $\epsilon$ is the Rossby number (which is small for
global tides), $f$ is the spatially-dependent non-dimensional Coriolis
parameter which is equal to the sine of the latitude (or which can be
approximated by a linear or constant profile for local area models),
$\beta$ is the Burger number (which is also small), $C$ is the (spatially varying) nondimensional drag
coefficient and $H$ is the (spatially varying) nondimensional fluid
depth at rest, and $\nabla$ and $\nabla\cdot$ are the intrinsic
gradient and divergence operators on the surface $\Omega$,
respectively.

Prior energy and error analysis in~\cite{CoKi} assumes that the bottom friction satisfies some $0 < C_* \leq C(\bfx) \leq C^*$.  The strict lower bound allows one to show an exponential damping of the energy.  However, the model is well-posed and, absent forcing, has non-increasing energy.  Since we are not working with energy estimates, it is sufficient for us to merely assume the upper bound $C(\bfx) \leq C^*$.  

As in~\cite{CoKi}, we arrive at a form suitable for discretization by mixed methods  by working with the linearized momentum
$\widetilde{u} = H u$ rather than velocity. After making this substitution and dropping the
tildes, we obtain
\begin{equation}
\begin{split}
\frac{1}{H}u_t + \frac{f}{H\epsilon} u^\perp + \frac{\beta}{\epsilon^2} \nabla
\eta  + \frac{C}{H} u & = F, \\
\eta_t + \nabla \cdot  u & = 0.
\end{split}
\label{eq:thepde}
\end{equation}
A natural weak formulation of this equations is to seek $u \in \hdiv$
and $\eta \in L^2$ so that
\begin{equation}
\begin{split}
\left( \frac{1}{H}u_t , v \right) 
+ \frac{1}{\epsilon} \left( \frac{f}{H} u^\perp , v \right) 
- \frac{\beta}{\epsilon^2} \left( \eta ,
\nabla \cdot v \right) + \left( \frac{C}{H} u , v \right) & = 
\left( F , v \right), \\
\left( \eta_t , w \right) + \left( \nabla \cdot u  , w \right)& = 0
\end{split}
\label{eq:mixed}
\end{equation}
for all $v \in \hdiv$ and $w \in L^2$.

We select suitable mixed finite element spaces $V_h \subset \hdiv$ and
$W_h \subset L^2$ of order $k$ satisfying the commuting projection
and having divergence mapping $V_h$ onto $W_h$~\cite{brezzi1991mixed}.  Since we are not proving error estimates in this paper, we do not recount the particulars of these projections.  However, we will need to make use of the \emph{inverse assumption} that there exists some $C_I$ (typically depending on the polynomial degree and mesh shape but not element size) such that
\begin{equation}
  \label{eq:inverse}
  \| \nabla \cdot u \| \leq \tfrac{C_I}{h} \| u \|
\end{equation}
for all $u \in V_h$.

Our examples will use the Raviart-Thomas~\cite{RavTho77a} triangular finite elements for
$V_h$ together with discontinuous piecewise polynomials for $W_h$.  We
follow the ordering of the Periodic Table of Finite Elements~\cite{arnold2014periodic}
summarize Finite Element Exterior Calculus~\cite{arnold2006finite} rather than the original
ordering of Raviart and Thomas so that the lowest order RT space is
$RT_1$ combined with $dP_0$.  We will also employ the recently-developed trimmed serendipity elements~\cite{gillette2019computational, gillette2019trimmed}.  These elements are smaller than the rectangular RT elements for the same order of approximation, and pair with $P_k$ rather than tensor-product spaces for $W_h$.

We define $u_h \subset V_h$
and $\eta_h \subset W_h$ as solutions of the discrete variational
problem
\begin{equation}
\begin{split}
\left( \frac{1}{H}u_{h,t} , v_h \right) 
+ \frac{1}{\epsilon} \left( \frac{f}{H} u_h^\perp , v_h \right) 
- \frac{\beta}{\epsilon^2} \left( \eta_h ,
\nabla \cdot v_h \right) + \left( \frac{C}{H} u_h , v_h \right) & =
\left( F , v_h \right)
, \\
\left( \eta_{h,t} , w_h \right) + \left( \nabla \cdot u_h  , w_h \right)& = 0.
\end{split}
\label{eq:discrete_mixed}
\end{equation}

In previous work~\cite{CoKi}, we analyzed the semi-discrete form of
this method, and in~\cite{kirbykieu} we analyzed a symplectic Euler time
discretization of mixed methods for the simpler (obtained from our current model putting $f=0, C=0$ and possibly allowing $\beta, \epsilon$ to vary spatially)
acoustic wave equation.  For each time step, this method only requires
the inversion of a mass matrix for $V_h$ and another for $W_h$.
However, it requires a CFL-like time step constraint with 
$\Delta t = \mathcal{O}(h)$ (the constant depends somehow on the shape
of mesh elements), is only first-order accurate, and only
conserves a quantity close to the actual system energy in the undamped
case.  Moreover, in the finite element context even explicit methods
require the inversion of mass matrices, unless the mesh
and approximating spaces admit some kind of diagonal approximation (e.g. lumping).

In this paper, we turn to implicit methods, especially
Crank-Nicolson.  This method is second-order accurate in time, does
not require a CFL condition for stability, and exactly conserves the
system energy for the undamped equations.   We also point out
that, for linear problems it is equivalent to the implicit midpoint
rule, which is the lowest-order Gauss-Legendre implicit Runge-Kutta
method.  In addition to their A-stability, these methods are both
symplectic and B-stable, which makes them seem quite appropriate for
problems based on a energy conservation principle plus some
damping mechanism.  (Note: for nonlinear problems, Crank-Nicolson is actually the lowest-order LobattoIIIA method, which is still A-stable but not symplectic.) On the down side, it requires the
solution of a more complicated system of equations at each time step than symplectic Euler. 
Error analysis goes through following standard
techniques; our goal here is the design and analysis of an effective
preconditioner.  

Selecting time levels $0 = t_0 < t_1 < \dots < t_N = T$ with
$t_n = t_0 + n \Delta t$, we seek a sequence of $\{ (u_h^n, \eta_h^n)
\}_{n=0}^N$ such that for each $n \geq 1$,
\begin{equation}
\begin{split}
\left( \frac{1}{H}\frac{u_h^{n+1}-u_h^n}{\Delta t} , v_h \right) 
+ \frac{1}{\epsilon}
\left( \frac{f}{2H} \left( (u_h^{n+1})^\perp + (u_h^n)^\perp \right) , v_h \right)   & \\
- \frac{\beta}{2\epsilon^2} \left( \eta_h^{n+1} + \eta_h^n ,
\nabla \cdot v_h \right) + \left( \frac{C}{2H} \left(u_h^{n+1} + u_h^n\right) , v_h \right) & =
\left( F^{n+\tfrac{1}{2}} , v_h \right)
, \\
\left( \frac{\eta_{h}^{n+1}-\eta_h^n}{\Delta t} , w_h \right) + \left(
\tfrac{1}{2} \nabla \cdot (u_h^{n+1} + u_h^n)  , w_h \right)& = 0.
\end{split}
\label{eq:cn_mixed}
\end{equation}
for all $v_h \in V_h$ and $w_h \in W_h$.  

Given $u_h^n$ and $\eta_h^n$, this defines a linear system for 
 $u_h^{n+1}$ and $\eta_h^{n+1}$ that must be solved at each time
step.  Dropping the superscripts and subscripts, multiplying
through by $\Delta t$ and putting $k \equiv \tfrac{\Delta t}{2}$, we arrive at a
canonical equation to be solved at each time step:

\begin{equation}
  \begin{split}
    \left( \frac{1}{H} u , v \right)
    + \left( \frac{f k}{\epsilon H} u^\perp , v \right)
  - \frac{\beta k}{\epsilon^2} \left( \eta, \nabla \cdot v \right)
  + \left( \frac{C k}{H} u , v \right) & = \left( F, v \right), \\
  \left( \eta , w \right) + k\left( \nabla \cdot u , w \right)
  & = \left( G , w \right),
  \end{split}
  \label{prec_me}
\end{equation}
where the solution $u \in V_h$ and $\eta \in W_h$ and similar for test functions.
Equivalently, we can define a bilinear form on the product space
$V_h \times W_h$.  Adding together the first equation and
$\tfrac{\beta}{\epsilon^2}$ times the second, we let  $\bfu = (u,
\eta)$ and $\bfv = (v, w)$, to define
\begin{equation}
  \label{eq:weakform}
  \begin{split}
  a(\bfu, \bfv) =  & \left( \frac{1}{H} u , v \right)
    + \left( \frac{f k}{\epsilon H} u^\perp , v \right)
    - \frac{\beta k}{\epsilon^2} \left( \eta, \nabla \cdot v \right) \\
    & 
  + \left( \frac{C k}{H} u , v \right)
  + \frac{\beta}{\epsilon^2}\left( \eta , w \right) + \frac{\beta
    k}{\epsilon^2} \left( \nabla \cdot u , w \right).
  \end{split}
\end{equation}
Before proceeding, we remark that other methods (e.g. backward Euler or the implicit midpoint rule) would give variational problems of this form as well.
Now, solving a variational problem associated with this bilinear form gives rise to a block-structured linear system
\begin{equation}
  \label{eq:blockmat}
  \begin{bmatrix}
    \check{M} & -\frac{\beta k}{\epsilon^2} D^T \\
    \frac{\beta k}{\epsilon^2} D & \frac{\beta}{\epsilon^2} M
  \end{bmatrix}
  \begin{bmatrix} \mathrm{u} \\ \mathrm{\eta} \end{bmatrix}
  = \begin{bmatrix} \mathrm{f} \\ \mathrm{g} \end{bmatrix},
\end{equation}
where for finite element bases $\{ \psi_i \}_{i=1}^{\dim V_h}$ and $\{ \phi_i \}_{i=1}^{\dim W_h}$, we have matrices
\begin{equation}
  \label{eq:mats}
  \begin{split}
  \check{M}_{ij} & = \left( \frac{1+C k}{H} \psi_j , \psi_i \right)
  + \left( \frac{f k}{\epsilon H} \psi_j^\perp , \psi_i \right),  \\
  D_{ij} & = \left( \nabla \cdot \psi_j, \phi_i \right), \\
  M_{ij} & = \left( \phi_j , \phi_i \right).
  \end{split}
\end{equation}
Note that $\check{M}$ is not just a weighted mass matrix.  It is nonsymmetric owing to skew-symmetric term above.  This skew term on the diagonal (rather than the off-diagonal blocks having a skew structure) that seems to lead to parameter-dendence later in our weighted-norm estimate.

\section{Preconditioning}
\label{sec:precond}
Now, we turn to developing a preconditioner for~\eqref{eq:weakform},~\eqref{eq:blockmat}.  Here, we concretize the abstract approach taken in~\cite{kirby2010functional,mardal2011preconditioning} for our particular tide model.  Essentially, a bounded bilinear form $a$ on a Hilbert space $V$ is equivalent to a linear operator $\mathcal{A}$ from $V$ into its topological dual $V^\prime$.  Classical Galerkin discretization restricts this bilinear form and operator to some finite-dimensional subspace $V_h \subset V$.  Moreover, the discrete operator $\mathcal{A}_h: V_h \rightarrow V_h^\prime$ is encoded by the usual finite element stiffness matrix $A$ obtained by substituting each member of a basis for $V_h$ into each argument of $a$.

When one seeks to solve the linear system for the discrete solution by means of an iterative method such as GMRES~\cite{saad1986gmres}, the \emph{conditioning} of the matrix $A$ plays a critical role.  As the condition number, and hence number of iterations required, of $A$ degrades under mesh refinement, it is critical to \emph{precondition} the linear system by means of (at least morally) pre-multiplying the system
\[
A x = b
\]
by some linear operator $P^{-1}$.  Thus, one obtains the equivalent system
\[
P^{-1} A x = P^{-1} b,
\]
and if the conditioning of $P^{-1} A$ is much better than that of $A$,
the iterative method should converge much faster.  Of course, the cost of applying $P^{-1}$ at each iteration must not offset the reduction in iteration count for the preconditioner to be successful.

One can think of the matrix $P$ as discretizing some simpler operator
$\mathcal{P}:V \rightarrow V^{\prime}$ so that the product $P^{-1} A$
encodes a bounded operator from $V_h$ onto itself.  In the simplest
case this is the \emph{Riesz map}, which isometrically identifies each
$f \in V_h^\prime$ uniquely with some $v \in V_h$ so that $F(u) = (u, v)$
for all $u \in V_h$.  Bounded operators have bounded spectra, and
functional-analytic bounds obtained on $\mathcal{P}^{-1} \mathcal{A}$
mean that the matrices $P^{-1} A$ will inherit mesh-independent bounds
on their spectra.  We refer to~\cite{kirby2010functional,
  mardal2011preconditioning} for further discussion of this approach.

In addition to mesh refinement, variation in physical parameters can also contribute adversely to the conditioning of discrete problems.  While the standard Riesz map serves as a simple preconditioner that eliminates mesh dependence, it does not address physical constants.  Increasingly, attempts are made to design \emph{parameter-robust} preconditioners, meaning that they also eliminate or at least mitigate the dependence of the conditioning on system parameters.  

In this section, we present two preconditioners.  One is based on
inverting weighted mass matrices.  This utilizes an inverse assumption
in $\hdiv$ to work in purely a discrete $L^2$ inner product, so that
the bounds depend on the mesh parameter $h$ in such a way that
conditioning (as expected) degrades as $h \searrow 0$.  However, this
dependence can be offset by taking $k = \mathcal{O}(h)$, a CFL-like
criterion that enforces conditioning rather than stability of the
time-discretization.  Our second approach better respects the
functional analytic structure, working in a weighted $\hdiv \times L^2$ inner product.  Here,
we obtain a mesh-independent bound that is also far less dependent on
other parameters at the expense of a more complicated operator to
invert as a preconditioner.  

\subsection{Mass matrices: block diagonal}
A simple approach that may help for small time steps is to
precondition the linear system with the block diagonal matrix
\begin{equation}
  \label{eq:Mprec}
  P_M = 
  \begin{bmatrix} \tilde{M} & 0 \\ 0 & \frac{\beta}{\epsilon^2} M \end{bmatrix},
\end{equation}
where
\begin{equation}
  \tilde{M}_{ij} = \left( \frac{1}{H} \psi_j, \psi_i \right)
\end{equation}
is the mass-like matrix obtained from the $\tfrac{1}{H}$-weighted inner product of the $V_h$ basis functions and $M$ is as in~\eqref{eq:mats}.

This is motivated by the observing that the bilinear form $a$
from~\eqref{eq:weakform} is continuous and coercive on discrete subspaces of
$(L^2)^2 \times L^2$, although the constants depend on the
discretization parameters $h$ and $k$ as well as the physical parameters.  We define the norm
\begin{equation}
  \| \bfu \|_2^2 = \| u \|_{\frac{1}{H}}^2 + \frac{\beta}{\epsilon^2} \| \eta \|^2,
\end{equation}
where $\| u \|_{\frac{1}{H}}^2 = \left(\tfrac{1}{H} u, u \right)$.  The
the inner product for this norm generates the matrices in~\eqref{eq:Mprec}.

Establishing well-posedness of variational problems for the bilinear form $a$ follows from demonstrating continuity and inf-sup estimates in $\hdiv \times L^2$.  However, we can study the mass matrix preconditioner~\eqref{eq:Mprec} by means of establishing continuity and coercivity of $a$ on finite element subspaces equipped with the $L^2$ norms.  This analysis is somewhat nonstandard, but it establishes an alternate proof of solvability of the discrete system and more importantly, allows us to demonstrate mesh-independence of~\eqref{eq:Mprec} as a preconditioner subject to a CFL-like restriction on $k$.

\begin{proposition}
  Let
  \begin{equation}
      \kappa = 4 \max\left\{ 1 + C^* k+ \frac{f^* k}{\epsilon},
      \frac{\sqrt{\beta} k}{\epsilon C_I H_* h} \right\}\\
  \end{equation}
  Then, the bilinear form $a$ satisfies
  \begin{equation}
    \label{eq:l2cont}
    a(\bfu, \bfv) \leq \kappa \| \bfu \|_2 \| \bfv \|_2.
  \end{equation}
  and
  \begin{equation}
    \label{eq:l2coercive}
    a(\bfu, \bfu) \geq \| \bfu \|^2_2
  \end{equation}
\end{proposition}
\begin{proof}
  We begin with the continuity estimate~\eqref{eq:l2cont}.
  \[
  \begin{split}
  a(\bfu, \bfv) & =  \left( \frac{1}{H} u , v \right)
    + \left( \frac{f k}{\epsilon H} u^\perp , v \right)
  - \frac{\beta k}{\epsilon^2} \left( \eta, \nabla \cdot v \right) \\
& \ \ \
  + \left( \frac{C k}{H} u , v \right)
  + \frac{\beta}{\epsilon^2}\left( \eta , w \right) + \frac{\beta
    k}{\epsilon^2} \left( \nabla \cdot u , w \right) \\
  & \leq \left( 1 + C^* k + \frac{f^* k}{\epsilon} \right) \| u
  \|_{\frac{1}{H}} \| v \|_{\frac{1}{H}} \\
  & \ \ \
  + \frac{\beta }{\epsilon^2} 
  \| \eta \| \| w \|
  + \frac{\beta k}{\epsilon^2}   \| \nabla \cdot u \| \| w \| + \frac{\beta k}{\epsilon^2} \| \eta \| \| \nabla \cdot v \| .
  \end{split}
  \]
  Applying the inverse estimate to the divergences and converting to
  the $\frac{1}{H}$-weighted norm now gives
  \[
  \begin{split}
    a(\bfu, \bfv) & \leq \left( 1 + C^* k + \frac{f^* k}{\epsilon} \right) \| u
  \|_{\frac{1}{H}} \| v \|_{\frac{1}{H}} \\
  & \ \ \
  + \frac{\beta }{\epsilon^2} 
  \| \eta \| \| w \|
  + \frac{\beta k}{C_I h H_*\epsilon^2}   \| u \|_{\frac{1}{H}} \| w \|
  + \frac{\beta k}{C_I h H_* \epsilon^2} \| \eta \| \| v \|_{\frac{1}{H}}.
  \end{split}
  \]
  The result follows by absorbing $\sqrt{\beta}/\epsilon$ into the norm of $\| w \|$ and $\| \eta \|$ in the third and fourth term and then recognizing each term as bounded by $\kappa \| \bfu \| \| \bfv \|$.

  The rescaling of the second equation to produce the bilinear form
  $a$ makes the coercivity estimate rather simple.  Noting that
  $u^\perp \cdot u =0$ pointwise and that the divergence terms in
  $a(\bfu, \bfu)$ cancel, we have
  \[
  \begin{split}
  a(\bfu, \bfu) & = \left( \frac{1}{H} u , u \right)
  + \left( \frac{C k}{H} u , u \right)
  + \frac{\beta k}{\epsilon^2}\left( \eta , \eta \right) \\
  & \geq \left( 1+C_* k \right) \| u \|_{\frac{1}{H}}^2 +
  \frac{\beta}{\epsilon}^2 \| \eta \|^2 \\
  & \geq \| \bfu \|^2_2.
  \end{split}
  \]
\end{proof}
It is possible to achieve slightly better constants (e.g. through more careful use of discrete Cauchy-Schwarz), but the main issue remains:  The conditioning of the system (continuity divided by coercivity constants) depends on the discretization (as well as physical) parameters, scaling like $\tfrac{k}{h}$.  For a fixed time step, the conditioning degrades like $h^{-1}$, and so preconditioning with weighted mass matrices is only scalable if one also imposes a CFL-like time step restriction.  Moreover, even including some weights in the norm, we still have quite a bit of parameter dependence in our estimate.

\subsection{Weighted-norm preconditioning}
The mesh-dependence in our estimate comes from invoking the inverse assumption in order to obtain $L^2$ estimates.   Our bilinear form is not coercive on subspaces of $\hdiv \times L^2$, but we can prove that it still defines a bounded operator with bounded inverse in a weighted norm that nearly eliminates parameter dependence.  Such techniques appear for other applications~\cite{adler2017robust, baerland2017weakly} as well, and are based on defining a suitable (parameter-dependent) inner product in which the problem is well-behaved rather than algebraic considerations such as merely selecting the block diagonal or triangular part of the system matrix~\cite{mardal2007order, wathen1993fast,Howle}

We can  equip $\hdiv$ with the following weighted norm
\begin{equation}
  \| u \|_{a}^2 = \| (1+C k)u \|_{\frac{1}{H}}^2 + \frac{k^2\beta}{\epsilon^2} \| \nabla \cdot u \|^2
\end{equation}
and, as previously, $L^2$ with the norm
\begin{equation}
  \| \eta \|_b^2 = \frac{\beta}{\epsilon^2} \| \eta \|^2.
\end{equation}
We then equip the product space $\hdiv \times L^2$ with the norm
\begin{equation}
  \begin{split}
    \opnorm{\bfu}^2 & = \opnorm{(u, \eta)}^2 = \| u \|_a^2 + \| \eta \|_b^2 \\
    & = \| (1+C k) u \|_{\frac{1}{H}}^2 + \frac{k^2\beta}{\epsilon^2} \| \nabla \cdot u \|^2 + \frac{\beta}{\epsilon^2} \| \eta \|^2
  \end{split}
\end{equation}

This norm is derived from a weighted inner product
\begin{equation}
  \label{eq:goodip}
  ((\bfu, \bfv)) =
  \left(\left(1+Ck\right)u, v\right)_{\frac{1}{H}}  + (u, v) + \frac{k^2 \beta}{\epsilon^2} \left(\nabla \cdot u, \nabla \cdot v\right) + \frac{\beta}{\epsilon^2} \left(\eta, w\right).
\end{equation}
Discretizing this bilinear form on mixed function spaces $V_h \times W_h$ yields a block-diagonal preconditioning matrix:
\begin{equation}
  \label{eq:goodp}
  P =  \begin{bmatrix}
    P_{V_h} & 0 \\ 0 & P_{W_h}
  \end{bmatrix},
\end{equation}
where the first block handles the parameter-weighted $\hdiv$ inner product and the second is just the standard $W_h$ mass matrix scaled by $\frac{\beta}{\epsilon^2}$.  We have
\begin{equation}
  \begin{split}
    (P_{V_h})_{ij} & = \left(\left(1+Ck\right)\psi_j, \psi_i\right)_{\frac{1}{H}}  + (\psi_j, \psi_i) + \frac{k^2 \beta}{\epsilon^2} \left(\nabla \cdot \psi_j, \nabla \cdot \psi_i\right), \\
    (P_{W_h})_{ij} & = \frac{\beta}{\epsilon^2} \left( \phi_j, \phi_i \right).
  \end{split}
\end{equation}

In the lowest-order case (either on triangles or squares), $W_h$ consists of piecewise constants so that $P_{W_h}$ is simply a diagonal matrix.  Since the top left block discretizes a differential operator, applying $P_{V_h}^{-1}$ will constitute the bulk of the cost in applying the preconditioner.  Options based on geometric or algebraic multigrid are available, and we discuss these more later.

The following result shows the boundedness of $a$ in this norm, with mild dependence on parameters, which we discuss below.
\begin{theorem}
  \label{thm:cont}
  For all $\bfu = (u, \eta)$ and $\bfv = (v, w)$ in $\hdiv \times L^2$, the bilinear form $a$ satisfies
  \begin{equation}
    a(\bfu, \bfv) \leq K \opnorm{\bfu} \opnorm{\bfv},
    \end{equation}
    where constant $K = K_{k,\epsilon} = \max\left\{2,1+\frac{k}{\epsilon}\right\}$.
\end{theorem}
\begin{proof}
  The proof is a direct calculation using Cauchy-Schwarz, the isometry of $\cdot^\perp$, and upper bounds on some of the spatially varying coefficients
  \begin{equation}
    \begin{split}
      a(\bfu, \bfv) & =
      \left( \frac{1}{H} u , v \right)
      + \left( \frac{f k}{\epsilon H} u^\perp , v \right)
      - \frac{\beta k}{\epsilon^2} \left( \eta, \nabla \cdot v \right) \\
      & \ \ \ \ 
      + \left( \frac{C k}{H} u , v \right)
      + \frac{\beta}{\epsilon^2}\left( \eta , w \right) + \frac{\beta
        k}{\epsilon^2} \left( \nabla \cdot u , w \right) \\
      & \leq \| (1+Ck) u \|_{\frac{1}{H}} \|(1+Ck) v \|_{\frac{1}{H}}
      + \frac{f^* k}{\epsilon} \| u \|_{\frac{1}{H}} \| v \|_{\frac{1}{H}} \\
      & \ \ \ \
      + \frac{\beta k}{\epsilon^2} \| \eta \| \| \nabla \cdot v \|
      + \frac{\beta}{\epsilon^2} \| \eta \| \| w \|
      + \frac{\beta k}{\epsilon^2} \| \nabla \cdot u \| \| w \|.
    \end{split}
  \end{equation}
  Now, we can write
  \[
  \| u \|_{\frac{1}{H}} \leq \frac{1}{\sqrt{1+C_* k}} \| (1+C k) u\|_{\frac{1}{H}}
  \leq \| (1+C k) u\|_{\frac{1}{H}}
  \]
  and recalling that $|f^*| \leq 1$,
  \begin{equation}
    \begin{split}
      a(\bfu, \bfv) &  \leq \left(1+\frac{k}{\epsilon} \right) \| (1+Ck) u \|_{\frac{1}{H}} \| (1+Ck) v \|_{\frac{1}{H}} \\
      & \ \ \ \
            + \frac{\beta k}{\epsilon^2} \| \eta \| \| \nabla \cdot v \|
      + \frac{\beta}{\epsilon^2} \| \eta \| \| w \|
      + \frac{\beta k}{\epsilon^2} \| \nabla \cdot u \| \| w \|.
  \end{split}
  \end{equation}
  Now, we recognize the right-hand side as the dot product of
  \[
  \begin{bmatrix}
    \sqrt{(1+\frac{k}{\epsilon})} \|(1+Ck) u\|_{\frac{1}{H}} \\
    \frac{\sqrt{\beta}}{\epsilon} \| \eta \| \\
    \frac{\sqrt{\beta}}{\epsilon} \| \eta \| \\
    \frac{k\sqrt{\beta}}{\epsilon} \| \nabla \cdot u \|
  \end{bmatrix}^t
  \begin{bmatrix}
    \sqrt{(1+\frac{k}{\epsilon})} \|(1+Ck) v\|_{\frac{1}{H}} \\
    \frac{k\sqrt{\beta}}{\epsilon} \| \nabla \cdot v \| \\
    \frac{\sqrt{\beta}}{\epsilon} \| w \| \\
    \frac{\sqrt{\beta}}{\epsilon} \| w \| 
  \end{bmatrix},
  \]
whence discrete Cauchy-Schwarz gives
\begin{equation}
  \begin{split}
    a(\bfu, \bfv) & \leq
  \left[ \left(1+\frac{k}{\epsilon}\right)\|(1+Ck)u\|_{\frac{1}{H}}^2
  + \frac{k^2 \beta}{\epsilon^2} \| \nabla \cdot u \|^2
  + \frac{2\beta}{\epsilon}^2 \| \eta \|^2
  \right]^{\frac{1}{2}} \\
      & \ \ \ \
    \times    
  \left[ \left(1+\frac{k}{\epsilon}\right)\|(1+Ck)v\|_{\frac{1}{H}}^2
  + \frac{k^2 \beta}{\epsilon^2} \| \nabla \cdot v \|^2
  + \frac{2\beta}{\epsilon}^2 \| w \|^2
  \right]^{\frac{1}{2}},
  \end{split}
\end{equation}
and the result follows from a simple bound.
\end{proof}

Note that the Coriolis term $\frac{f k}{\epsilon} (u^\perp, v)$, which is skew and on the diagonal leaves the term scaled by $\frac{k}{\epsilon}$ so that we do not obtain total parameter-independence.  We can interpret this bound as saying that the Rossby number $\epsilon$ induces a time scale, independent of $h$, that must be resolved in order to obtain a robust continuity estimate.  More precisely,
\begin{corollary}
  For any $M \geq 2$ and $\epsilon > 0$, there exists $k_0$ such that for any $k \leq k_0$,
  \[
  a(\bfu, \bfv) \leq M \opnorm{\bfu} \opnorm{\bfv}.
  \]
\end{corollary}

Next, we bound the inverse of the operator induced by $a$ by means of
an inf-sup condition.  Unlike our continuity estimate, this is completely parameter-independent.
\begin{theorem}
  \label{thm:infsup}
  The bilinear form $a$ satisfies the estimate
  \begin{equation}
    \inf_{\bfu \neq 0} \sup_{\bfv} a(\bfu, \bfv) \geq \frac{\sqrt{3}}{6}.
  \end{equation}
\end{theorem}
\begin{proof}
  We let $\bfu =(u, \eta)$ be given and put $\bfv = (v, w) = (u, \eta + k \nabla \cdot u)$ so that
  \begin{equation}
    \begin{split}
    a(\bfu, \bfv) & =  \left( \frac{1}{H} u , u \right)
    + \left( \frac{f k}{\epsilon H} u^\perp , u \right)
  - \frac{\beta k}{\epsilon^2} \left( \eta, \nabla \cdot u \right) \\
& \ \ \ \  + \left( \frac{C k}{H} u , u \right)
  + \frac{\beta}{\epsilon^2}\left( \eta , \eta + k \nabla \cdot u \right) + \frac{\beta
    k}{\epsilon^2} \left( \nabla \cdot u , \eta + k \nabla \cdot u \right) \\
  & = \| (1+ Ck) u \|_{\frac{1}{H}}^2 + \frac{\beta}{\epsilon^2} \| \eta \|^2
  + \frac{k^2\beta}{\epsilon^2} \| \nabla \cdot u \|^2 + \frac{k\beta}{\epsilon^2} (\eta, \nabla \cdot u). 
  \end{split}
  \end{equation}
  The last term is readily bounded below
  by $-\tfrac{\beta}{2\epsilon^2} \left( \| \eta \|^2 + k^2 \| \nabla \cdot u \|^2 \right)$ so that
  \begin{equation}
    \label{belowu2}
      a(\bfu, \bfv)  \geq \| (1+ Ck) u \|_{\frac{1}{H}}^2 + \frac{\beta}{2 \epsilon^2} \| \eta \|^2 + \frac{k^2\beta}{2 \epsilon^2} \| \nabla \cdot u \|^2      \geq \frac{1}{2} \opnorm{\bfu}^2.
  \end{equation}
  Now, we have that
  \begin{equation}
    \begin{split}
      \opnorm{\bfv}^2 & = \|(1+Ck) u\|^2_{\frac{1}{H}} + \frac{k^2 \beta}{\epsilon^2} \| \nabla \cdot u \|^2 + \frac{\beta}{\epsilon^2} \| \eta + k \nabla \cdot u \|^2 \\
      & \leq \| (1+Ck) u \|^2_{\frac{1}{H}} + 3\frac{\beta}{\epsilon^2} \| \nabla \cdot u \|^2 + 2\frac{\beta}{\epsilon^2} \| \eta \|^2 \\
      & \leq 3 \opnorm{\bfu}^2,
    \end{split}
  \end{equation}
  and combining this with~\eqref{belowu2} gives the result.
\end{proof}

Because the spectral radius of a matrix is bounded above by any natural norm, 
these results prove that the eigenvalues of $P^{-1} A$ are bounded below by a constant (in fact, $\tfrac{\sqrt{3}}{6}$) independently of the mesh size and all the physical constants.  The eigenvalues of $P^{-1} A$ are further bounded above the greater of 2 and $1+\frac{k}{\epsilon}$, which can degrade as the Rossby number decreases.   

\subsection{Dropping the damping term from the preconditioner}
In~\cite{cotter2018mixed}, we consider energy and error analysis of a possibly degenerate nonlinear damping term, where
the term $Cu$ in~\eqref{eq:thepde} is replaced by a more general
$g(u)$.  Typical use cases have a power law such as $g(u) = |u|^{p-1} u$, modified to have linear growth for large $u$ (at least as a technical assumption).  In this case, $g(u)$ tends to zero as $|u|$ does so that the effective damping decays.  

Carrying out the same manipulations that leads to~\eqref{eq:weakform} for nonlinear damping leads to the nonlinear variational form
\begin{equation}
  \label{eq:nonlinweakform}
  \begin{split}
  F(\bfu; \bfv) & =  \left( \frac{1}{H} u , v \right)
    + \left( \frac{f k}{\epsilon H} u^\perp , v \right)
    - \frac{\beta k}{\epsilon^2} \left( \eta, \nabla \cdot v \right) \\
    & \ \ \ 
  + \left( \frac{k}{H} g(u) , v \right)
  + \frac{\beta}{\epsilon^2}\left( \eta , w \right) + \frac{\beta
    k}{\epsilon^2} \left( \nabla \cdot u , w \right).
  \end{split}
\end{equation}

Newton-type methods require the Jacobian of this system.  Linearizing
about some state $\bfu_0=(u_0, \eta_0)$, we have
\begin{equation}
  \label{eq:Jac}
  \begin{split}
  J_{\bfu_0}(\bfu; \bfv) & =  \left( \frac{1}{H} u , v \right)
    + \left( \frac{f k}{\epsilon H} u^\perp , v \right)
  - \frac{\beta k}{\epsilon^2} \left( \eta, \nabla \cdot v \right) \\
& \ \ \  + \left( \frac{k}{H} g^\prime(u_0)u , v \right)
  + \frac{\beta}{\epsilon^2}\left( \eta , w \right) + \frac{\beta
    k}{\epsilon^2} \left( \nabla \cdot u , w \right).
  \end{split}
\end{equation}

All of the analysis carried out in~\cite{cotter2018mixed} required monotonicity of $g$, so that $g^\prime > 0$.  In this case,~\eqref{eq:Jac} takes the same form
as~\eqref{eq:weakform} with $C \leftrightarrow g^\prime(u_0)$.  As a result, our theory carries over directly to preconditioning each Newton step provided that the Riesz map~\eqref{eq:goodip} is updated at each iteration (of each time step).

On the other hand, many preconditioners such as algebraic multigrid can be relatively expensive to initialize, so that it is helpful to reuse the same bilinear form between successive linear solves as the damping changes.  We drop the damping term in the bilinear form in~\eqref{eq:goodip} to define
\begin{equation}
  \label{eq:lessgoodip}
  ((\bfu, \bfv))_* =
  (u, v)_{\frac{1}{H}}  + \frac{k^2 \beta}{\epsilon^2} (\nabla \cdot u, \nabla \cdot v) + \frac{\beta}{\epsilon^2} (\eta, w),
\end{equation}
and an associated norm
\begin{equation}
  \label{eq:lessgoodnorm}
  \opnorm{\bfu}_*^2 \equiv ((\bfu, \bfu)).
\end{equation}
This norm is, at the cost of some dependence on $C^*$, equivalent to $\opnorm{\cdot}$
\begin{proposition}
  For all $\bfu = (u, \eta) \in V_h \times W_h$,
  \begin{equation}
    \tfrac{1}{1+C^* k} \opnorm{\bfu}^2 \leq \opnorm{\bfu}_*^2 \leq \opnorm{\bfu}^2  
  \end{equation}
\end{proposition}
\begin{proof}
  The proof is elementary and uses that $\tfrac{1+Ck}{1+C^*k} \leq 1 \leq 1+Ck$ in the definition of $((\cdot, \cdot))$. 
\end{proof}
Theorems~\ref{thm:cont} and~\ref{thm:infsup} can be readily restated using this norm:
\begin{corollary}
  For all $\bfu, \bfv \in \hdiv \times L^2$,
  \begin{equation}
    a(\bfu, \bfv) \leq K_* \opnorm{\bfu}_* \opnorm{\bfv}_*,
  \end{equation}
  where $K_* = (1+C^*k) \max\{2, 1+\tfrac{k}{\epsilon} \}$, and the inf-sup constant of $a$ with respect to $\opnorm{\cdot}_*$ is also at least $\tfrac{\sqrt{3}}{6}$.
\end{corollary}
Typically, the linear damping is small compared to the other effects in the equation so that the effective bounds on the preconditioner are essentially unchanged.  Much as with the Rossby number, the time step can be reduced to accommodate large $C^*$ if it becomes a problem.

\section{Numerical results}
\label{sec:numres}
We have implemented a mixed finite element discretization of the tide model and developed all of our preconditioners within the Firedrake framework~\cite{Rathgeber:2016}.  Firedrake is an automated system for the solution of PDE using the finite element method.  It allows users to specify the variational form of their problems using the Unifed Form Language (UFL) in Python~\cite{alnaes2014unified}, generates efficient low-level code for the evaluation of operators, and interfaces tightly with PETSc for scalable algebraic solvers.  Firedrake also allows users to specify UFL for preconditioning operator that is distinct from that for the problem being solved, and we make use of this facility.  A sample listing is shown in Figure~\ref{fig:code}.

\begin{figure}[htbp]
  \centering
  \begin{lstlisting}
    from firedrake import *

    mesh = UnitSquareMesh(16, 16)
    V = FunctionSpace(mesh, "RT", 1)
    Q = FunctionSpace(mesh, "DG", 0)

    Z = V * Q

    x, y = SpatialCoordinate(mesh)
    k = Constant(k)
    Eps = Constant(eps)
    Beta = Constant(0.1)
    C = Constant(1.0)
    f = Constant(1.0)
    beps2 = Beta / Eps**2

    up = Function(Z)
    v, q = TestFunctions(Z)

    u, p = split(up)

    F = (inner(u, v) * dx
        + k / Eps * f * inner(perp(u),v) * dx
        - k * beps2 * inner(p, div(v)) * dx
        + C * k * inner(u,v) * dx
        + beps2 * inner(p, q) * dx
        + k * beps2 * inner(div(u), q) * dx
        - beps2 * inner(sin(pi*x)*cos(pi*y),q)*dx)

    uu, pp = TrialFunctions(Z)
    Jpc = ((Constant(1.0) + C * k) * inner(uu,v)*dx
            + k**2 * beps2 * inner(div(uu),div(v)) *dx
            + beps2 * inner(pp,q)*dx)
        
    bcs = [DirichletBC(Z.sub(0), 0, 'on_boundary')]

    solve(F==0, up, bcs=bcs, Jp=Jpc)
    solver.solve()
  \end{lstlisting}
  \caption{Sample Firedrake code for solving the tide model using the Riesz map~\eqref{eq:goodip} as a preconditioner.  The user can optionally pass a Python dictionary containing PETSc options into the solve function.}
  \label{fig:code}
\end{figure}

In our initial experiments, we test the results of Theorems~\ref{thm:cont} and~\ref{thm:infsup}.  We test our methods on a simple square domain and put damping coefficient $C$, Coriolis parameter $f$, and bathymetry $H$ all equal 1, the Burger number $\beta=0.1$ and Rossby number $\epsilon = 0.1$.  Initially, we divide the unit square into an $N \times N$ mesh of squares, each subdivided into two right triangles, and use $V_h \times W_h$ as the lowest-order Raviart-Thomas and discontinuous piecewise constant spaces.  We vary the mesh size $h$, and the time step $k$ over relatively wide ranges, as shown in the figures below.  In each case, we solve the linear system using GMRES preconditioned with $P$ from~\eqref{eq:goodp} as well as the discretization of the simplified inner product in~\eqref{eq:lessgoodip}.  To verify the efficacy of each $P$, we simply apply $P^{-1}$ using a sparse direct factorization at first (see Figure~\ref{fig:opts1} for the solver dictionary we pass to \texttt{solve} in that case).   These results are plotted in Figure~\ref{fig:rt1trivaryingmesh}.  We observe that for each $k$, the iteration count seems to be bounded above independent of the mesh size.  For the moderately-sized $k$, the iteration count seems to be larger than for either very small or very large $k$, which means the iterations also seem to vary beneath a $k$-independent bound as well.  Comparing Figure~\ref{rt1trihp} to~\ref{rt1trihpred}, we see that removing the damping term from the preconditioner leads to a possible slight increase in iteration count.

\begin{figure}[htbp]
  \centering
  \begin{lstlisting}
param_lu = {"mat_type": "aij",
            "snes_type": "ksponly",
            "ksp_type": "gmres",
            "ksp_gmres_restart": 100,
            "pc_type": "lu"
            }
  \end{lstlisting}
  \caption{Firedrake solver parameter dictionary, internally mapped to PETSc options, indicating that the matrices will be assembled in standard sparse format, that the problem is linear (bypassing Newton), setting GMRES as the Krylov solver, and applying the inverse of the preconditioning matrix via LU factorization.}
  \label{fig:opts1}
\end{figure}

Nothing in our analysis depended on the particularities of the
approximating space, and we repeated the experiment for
$RT_2 \times dP_1$~(Figure~\ref{fig:rt2trivaryingmesh}) and
$RTc_1 \times dQ_0$~(Figure~\ref{fig:rt1quadvaryingmesh}).  The results
have the same flavor and differ only slightly in the particular
iteration counts compared to Figure~\ref{fig:rt1trivaryingmesh}.

\begin{figure}[h]
        \begin{subfigure}[l]{0.475\textwidth}
                \begin{tikzpicture}[scale=0.88]
                \begin{semilogxaxis}[xlabel={$N$}, ylabel={Iterations},
                log basis x=2,
                ylabel near ticks,
                x tick label style={font=\tiny}]
                \addplot[dashed,mark=square*,mark options={solid,fill=gray}] table [x = N, y = iterations, col sep=comma]
                {tri.Riesz.deg1.k1.0.csv};
                \addplot[dashdotted,mark=triangle,mark options={solid}] table [x = N, y = iterations, col sep=comma]
                {tri.Riesz.deg1.k0.1.csv};
                \addplot[dotted,mark=square*, mark options={solid,fill}] table [x = N, y = iterations, col sep=comma]
                {tri.Riesz.deg1.k0.01.csv};
                \addplot[dash dot dot,mark=triangle, mark options={solid,fill}] table [x = N, y = iterations, col sep=comma]
                {tri.Riesz.deg1.k0.001.csv};
                \addplot[dotted,mark=square, mark options={solid}] table [x = N, y = iterations, col sep=comma]
                {tri.Riesz.deg1.k0.0001.csv};
                \addplot[dashdotted,mark=x, mark options={solid}] table [x = N, y = iterations, col sep=comma]
                {tri.Riesz.deg1.k1e-05.csv};
                \addplot[dash dot dot,mark=o, mark options={solid,fill}] table [x = N, y = iterations, col sep=comma]
                {tri.Riesz.deg1.k1e-06.csv};
                \end{semilogxaxis}
                \end{tikzpicture}
                \caption{Using the bilinear form~\eqref{eq:goodip} (includes damping) as a preconditioner}
                \label{rt1trihp}
       \end{subfigure}
        \hspace{0.04\textwidth}
        \begin{subfigure}[r]{0.475\textwidth}
                \begin{tikzpicture}[scale=0.88]
                \begin{semilogxaxis}[xlabel={$N$},
                log basis x=2,
                ylabel near ticks,
                x tick label style={font=\tiny},
                legend cell align=left,
                legend style={overlay, at={(-0.16, 1.05)}, anchor=south},
                legend columns=9]
                \addlegendimage{empty legend}
                \addlegendentry[text width=20pt,text depth=]
                {$k =$};
                \addplot[dashed,mark=square*,mark options={solid,fill=gray}] table [x = N, y = iterations, col sep=comma]
                {tri.Riesz_Lite.deg1.k1.0.csv};
                \addlegendentry{$10^0$}
                \addplot[dashdotted,mark=triangle,mark options={solid}] table [x = N, y = iterations, col sep=comma]
                {tri.Riesz_Lite.deg1.k0.1.csv};
                \addlegendentry{$10^{-1}$}
                \addplot[dotted,mark=square*, mark options={solid,fill}] table [x = N, y = iterations, col sep=comma]
                {tri.Riesz_Lite.deg1.k0.01.csv};
                \addlegendentry{$10^{-2}$}
                \addplot[dash dot dot,mark=triangle, mark options={solid,fill}] table [x = N, y = iterations, col sep=comma]
                {tri.Riesz_Lite.deg1.k0.001.csv};
                \addlegendentry{$10^{-3}$}
                \addplot[dotted,mark=square, mark options={solid}] table [x = N, y = iterations, col sep=comma]
                {tri.Riesz_Lite.deg1.k0.0001.csv};
                \addlegendentry{$10^{-4}$}
                \addplot[dashdotted,mark=x, mark options={solid}] table [x = N, y = iterations, col sep=comma]
                {tri.Riesz_Lite.deg1.k1e-05.csv};
                \addlegendentry{$10^{-5}$}
                \addplot[dash dot dot,mark=o, mark options={solid,fill}] table [x = N, y = iterations, col sep=comma]
               {tri.Riesz_Lite.deg1.k1e-06.csv};
                \addlegendentry{$10^{-6}$}
                \end{semilogxaxis}
                \end{tikzpicture}
                \caption{Using the bilinear form~\eqref{eq:lessgoodip} (without damping) as a preconditioner}
                \label{rt1trihpred}
        \end{subfigure}
        \caption{Iteration count versus mesh refinement under various $k$ values for $C=f=1$, $\beta=0.1$, and $\epsilon=0.01$.  The unit square is divided into an $N \times N$ mesh of squares, each subdivided into two right triangles.  Lowest-order Raviart-Thomas discretization is used. The iteration counts are largest for moderate $k$ and decrease as $k$ is either very large or small.  Also, removing the damping term (right) from the weighted inner product leads to a small increase in iteration count. }  
        \label{fig:rt1trivaryingmesh}
\end{figure}
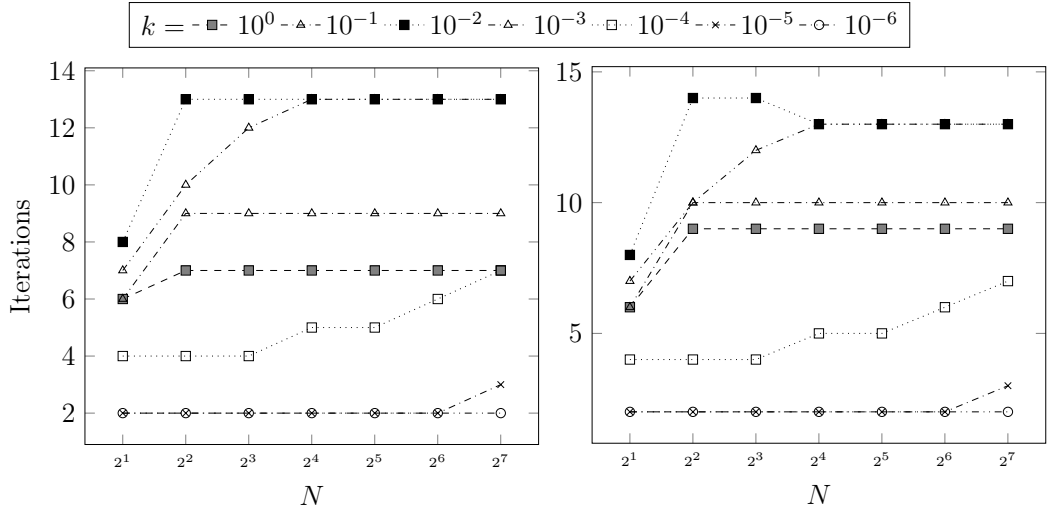

\vspace{1.0in}

\begin{figure}[h]
        \begin{subfigure}[l]{0.475\textwidth}
                \begin{tikzpicture}[scale=0.88]
                \begin{semilogxaxis}[xlabel={$N$}, ylabel={Iterations},
                log basis x=2,
                ylabel near ticks,
                x tick label style={font=\tiny}]
                \addplot[dashed,mark=square*,mark options={solid,fill=gray}] table [x = N, y = iterations, col sep=comma]
                {tri.Riesz.deg2.k1.0.csv};
                \addplot[dashdotted,mark=triangle,mark options={solid}] table [x = N, y = iterations, col sep=comma]
                {tri.Riesz.deg2.k0.1.csv};
                \addplot[dotted,mark=square*, mark options={solid,fill}] table [x = N, y = iterations, col sep=comma]
                {tri.Riesz.deg2.k0.01.csv};
                \addplot[dash dot dot,mark=triangle, mark options={solid,fill}] table [x = N, y = iterations, col sep=comma]
                {tri.Riesz.deg2.k0.001.csv};
                \addplot[dotted,mark=square, mark options={solid}] table [x = N, y = iterations, col sep=comma]
                {tri.Riesz.deg2.k0.0001.csv};
                \addplot[dashdotted,mark=x, mark options={solid}] table [x = N, y = iterations, col sep=comma]
                {tri.Riesz.deg2.k1e-05.csv};
                \addplot[dash dot dot,mark=o, mark options={solid,fill}] table [x = N, y = iterations, col sep=comma]
                {tri.Riesz.deg2.k1e-06.csv};
                \end{semilogxaxis}
                \end{tikzpicture}
                \caption{Using the bilinear form~\eqref{eq:goodip} (includes damping) as a preconditioner}
       \end{subfigure}
        \hspace{0.04\textwidth}
        \begin{subfigure}[r]{0.475\textwidth}
                \begin{tikzpicture}[scale=0.88]
                \begin{semilogxaxis}[xlabel={$N$},
                log basis x=2,
                ylabel near ticks,
                x tick label style={font=\tiny},
                legend cell align=left,
                legend style={overlay, at={(-0.16, 1.05)}, anchor=south},
                legend columns=9]
                \addlegendimage{empty legend}
                \addlegendentry[text width=20pt,text depth=]
                {$k =$};
                \addplot[dashed,mark=square*,mark options={solid,fill=gray}] table [x = N, y = iterations, col sep=comma]
                {tri.Riesz_Lite.deg2.k1.0.csv};
                \addlegendentry{$10^0$}
                \addplot[dashdotted,mark=triangle,mark options={solid}] table [x = N, y = iterations, col sep=comma]
                {tri.Riesz_Lite.deg2.k0.1.csv};
                \addlegendentry{$10^{-1}$}
                \addplot[dotted,mark=square*, mark options={solid,fill}] table [x = N, y = iterations, col sep=comma]
                {tri.Riesz_Lite.deg2.k0.01.csv};
                \addlegendentry{$10^{-2}$}
                \addplot[dash dot dot,mark=triangle, mark options={solid,fill}] table [x = N, y = iterations, col sep=comma]
                {tri.Riesz_Lite.deg2.k0.001.csv};
                \addlegendentry{$10^{-3}$}
                \addplot[dotted,mark=square, mark options={solid}] table [x = N, y = iterations, col sep=comma]
                {tri.Riesz_Lite.deg2.k0.0001.csv};
                \addlegendentry{$10^{-4}$}
                \addplot[dashdotted,mark=x, mark options={solid}] table [x = N, y = iterations, col sep=comma]
                {tri.Riesz_Lite.deg2.k1e-05.csv};
                \addlegendentry{$10^{-5}$}
                \addplot[dash dot dot,mark=o, mark options={solid,fill}] table [x = N, y = iterations, col sep=comma]
               {tri.Riesz_Lite.deg2.k1e-06.csv};
                \addlegendentry{$10^{-6}$}
                \end{semilogxaxis}
                \end{tikzpicture}
                \caption{Using the bilinear form~\eqref{eq:lessgoodip} (without damping) as a preconditioner}
        \end{subfigure}
        \caption{Experiment in Figure~\ref{fig:rt1trivaryingmesh} is repeated, except with the next-to-lowest Raviart-Thomas elements.  Since the bilinear form~\eqref{eq:goodip} is also discretized in this space, very little changes relative to the lowest-order case.}  
        \label{fig:rt2trivaryingmesh}
\end{figure}
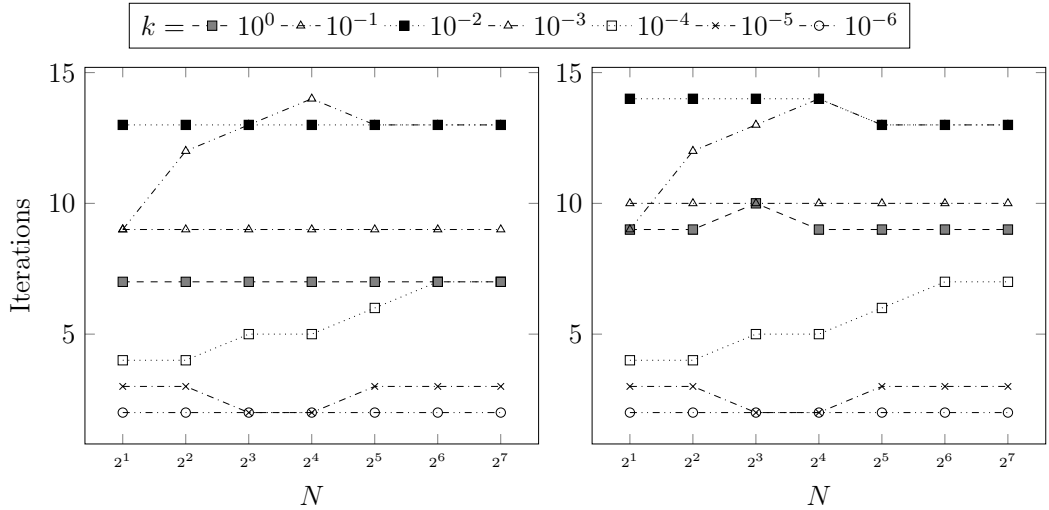

\begin{figure}[h]
        \begin{subfigure}[l]{0.475\textwidth}
                \begin{tikzpicture}[scale=0.88]
                \begin{semilogxaxis}[xlabel={$N$}, ylabel={Iterations},
                log basis x=2,
                ylabel near ticks,
                x tick label style={font=\tiny}]
                \addplot[dashed,mark=square*,mark options={solid,fill=gray}] table [x = N, y = iterations, col sep=comma]
                {tri.Riesz.deg2.k1.0.csv};
                \addplot[dashdotted,mark=triangle,mark options={solid}] table [x = N, y = iterations, col sep=comma]
                {tri.Riesz.deg2.k0.1.csv};
                \addplot[dotted,mark=square*, mark options={solid,fill}] table [x = N, y = iterations, col sep=comma]
                {tri.Riesz.deg2.k0.01.csv};
                \addplot[dash dot dot,mark=triangle, mark options={solid,fill}] table [x = N, y = iterations, col sep=comma]
                {tri.Riesz.deg2.k0.001.csv};
                \addplot[dotted,mark=square, mark options={solid}] table [x = N, y = iterations, col sep=comma]
                {tri.Riesz.deg2.k0.0001.csv};
                \addplot[dashdotted,mark=x, mark options={solid}] table [x = N, y = iterations, col sep=comma]
                {tri.Riesz.deg2.k1e-05.csv};
                \addplot[dash dot dot,mark=o, mark options={solid,fill}] table [x = N, y = iterations, col sep=comma]
                {tri.Riesz.deg2.k1e-06.csv};
                \end{semilogxaxis}
                \end{tikzpicture}
                \caption{Using the bilinear form~\eqref{eq:goodip} (includes damping) as a preconditioner}
       \end{subfigure}
        \hspace{0.04\textwidth}
        \begin{subfigure}[r]{0.475\textwidth}
                \begin{tikzpicture}[scale=0.88]
                \begin{semilogxaxis}[xlabel={$N$},
                log basis x=2,
                ylabel near ticks,
                x tick label style={font=\tiny},
                legend cell align=left,
                legend style={overlay, at={(-0.16, 1.05)}, anchor=south},
                legend columns=9]
                \addlegendimage{empty legend}
                \addlegendentry[text width=20pt,text depth=]
                {$k =$};
                \addplot[dashed,mark=square*,mark options={solid,fill=gray}] table [x = N, y = iterations, col sep=comma]
                {tri.Riesz_Lite.deg2.k1.0.csv};
                \addlegendentry{$10^0$}
                \addplot[dashdotted,mark=triangle,mark options={solid}] table [x = N, y = iterations, col sep=comma]
                {tri.Riesz_Lite.deg2.k0.1.csv};
                \addlegendentry{$10^{-1}$}
                \addplot[dotted,mark=square*, mark options={solid,fill}] table [x = N, y = iterations, col sep=comma]
                {tri.Riesz_Lite.deg2.k0.01.csv};
                \addlegendentry{$10^{-2}$}
                \addplot[dash dot dot,mark=triangle, mark options={solid,fill}] table [x = N, y = iterations, col sep=comma]
                {tri.Riesz_Lite.deg2.k0.001.csv};
                \addlegendentry{$10^{-3}$}
                \addplot[dotted,mark=square, mark options={solid}] table [x = N, y = iterations, col sep=comma]
                {tri.Riesz_Lite.deg2.k0.0001.csv};
                \addlegendentry{$10^{-4}$}
                \addplot[dashdotted,mark=x, mark options={solid}] table [x = N, y = iterations, col sep=comma]
                {tri.Riesz_Lite.deg2.k1e-05.csv};
                \addlegendentry{$10^{-5}$}
                \addplot[dash dot dot,mark=o, mark options={solid,fill}] table [x = N, y = iterations, col sep=comma]
               {tri.Riesz_Lite.deg2.k1e-06.csv};
                \addlegendentry{$10^{-6}$}
                \end{semilogxaxis}
                \end{tikzpicture}
                \caption{Using the bilinear form~\eqref{eq:lessgoodip} (without damping) as a preconditioner}
        \end{subfigure}
        \caption{Experiment in Figure~\ref{fig:rt1trivaryingmesh} is again repeated, except with lowest Raviart-Thomas elements on squares.  Again, little changes relative to the triangular case.}  
        \label{fig:rt1quadvaryingmesh}
\end{figure}
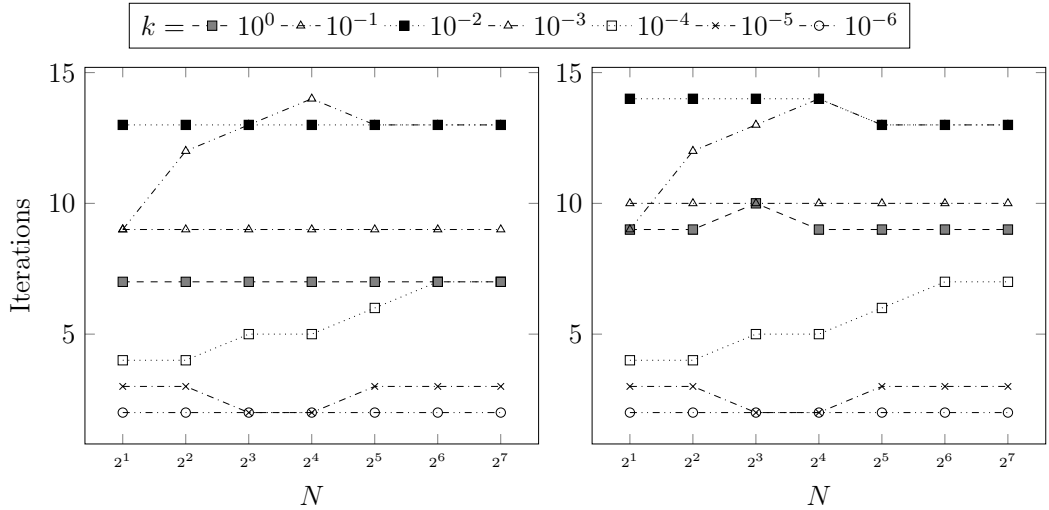

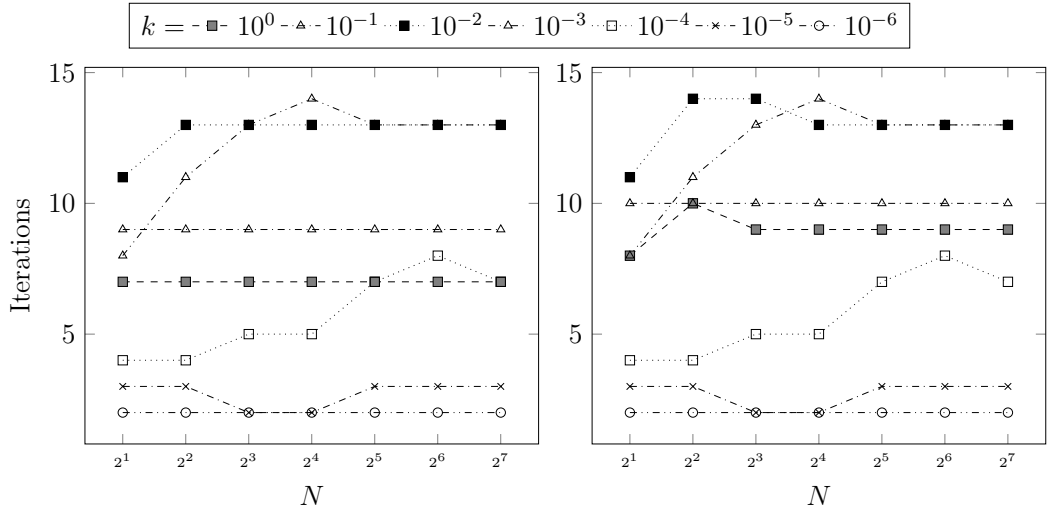
\begin{figure}[h]
        \begin{subfigure}[l]{0.475\textwidth}
                \begin{tikzpicture}[scale=0.88]
                \begin{semilogxaxis}[xlabel={$N$}, ylabel={Iterations},
                log basis x=2,
                ylabel near ticks,
                x tick label style={font=\tiny}]
                \addplot[dashed,mark=square*,mark options={solid,fill=gray}] table [x = N, y = iterations, col sep=comma]
                {Sminus.Riesz.deg2.k1.0.csv};
                \addplot[dashdotted,mark=triangle,mark options={solid}] table [x = N, y = iterations, col sep=comma]
                {Sminus.Riesz.deg2.k0.1.csv};
                \addplot[dotted,mark=square*, mark options={solid,fill}] table [x = N, y = iterations, col sep=comma]
                {Sminus.Riesz.deg2.k0.01.csv};
                \addplot[dash dot dot,mark=triangle, mark options={solid,fill}] table [x = N, y = iterations, col sep=comma]
                {Sminus.Riesz.deg2.k0.001.csv};
                \addplot[dotted,mark=square, mark options={solid}] table [x = N, y = iterations, col sep=comma]
                {Sminus.Riesz.deg2.k0.0001.csv};
                \addplot[dashdotted,mark=x, mark options={solid}] table [x = N, y = iterations, col sep=comma]
                {Sminus.Riesz.deg2.k1e-05.csv};
                \addplot[dash dot dot,mark=o, mark options={solid,fill}] table [x = N, y = iterations, col sep=comma]
                {Sminus.Riesz.deg2.k1e-06.csv};
                \end{semilogxaxis}
                \end{tikzpicture}
                \caption{Using the bilinear form~\eqref{eq:goodip} (includes damping) as a preconditioner}
       \end{subfigure}
        \hspace{0.04\textwidth}
        \begin{subfigure}[r]{0.475\textwidth}
                \begin{tikzpicture}[scale=0.88]
                \begin{semilogxaxis}[xlabel={$N$},
                log basis x=2,
                ylabel near ticks,
                x tick label style={font=\tiny},
                legend cell align=left,
                legend style={overlay, at={(-0.16, 1.05)}, anchor=south},
                legend columns=9]
                \addlegendimage{empty legend}
                \addlegendentry[text width=20pt,text depth=]
                {$k =$};
                \addplot[dashed,mark=square*,mark options={solid,fill=gray}] table [x = N, y = iterations, col sep=comma]
                {Sminus.Riesz_Lite.deg2.k1.0.csv};
                \addlegendentry{$10^0$}
                \addplot[dashdotted,mark=triangle,mark options={solid}] table [x = N, y = iterations, col sep=comma]
                {Sminus.Riesz_Lite.deg2.k0.1.csv};
                \addlegendentry{$10^{-1}$}
                \addplot[dotted,mark=square*, mark options={solid,fill}] table [x = N, y = iterations, col sep=comma]
                {Sminus.Riesz_Lite.deg2.k0.01.csv};
                \addlegendentry{$10^{-2}$}
                \addplot[dash dot dot,mark=triangle, mark options={solid,fill}] table [x = N, y = iterations, col sep=comma]
                {Sminus.Riesz_Lite.deg2.k0.001.csv};
                \addlegendentry{$10^{-3}$}
                \addplot[dotted,mark=square, mark options={solid}] table [x = N, y = iterations, col sep=comma]
                {Sminus.Riesz_Lite.deg2.k0.0001.csv};
                \addlegendentry{$10^{-4}$}
                \addplot[dashdotted,mark=x, mark options={solid}] table [x = N, y = iterations, col sep=comma]
                {Sminus.Riesz_Lite.deg2.k1e-05.csv};
                \addlegendentry{$10^{-5}$}
                \addplot[dash dot dot,mark=o, mark options={solid,fill}] table [x = N, y = iterations, col sep=comma]
               {Sminus.Riesz_Lite.deg2.k1e-06.csv};
                \addlegendentry{$10^{-6}$}
                \end{semilogxaxis}
                \end{tikzpicture}
                \caption{Using the bilinear form~\eqref{eq:lessgoodip} (without damping) as a preconditioner}
        \end{subfigure}
        \caption{Experiment in Figure~\ref{fig:rt1trivaryingmesh} is again repeated, except second-order trimmed serendipity elements on squares.  Again, little changes relative to the triangular case.}  
        \label{fig:sminus2quadvaryingmesh}
\end{figure}

We perform a second set of experiments, now fixing the mesh at $N=128$ and studying the iteration count as function of $\epsilon$ and $k$.  These results are shown in Figures~\ref{fig:rt1triepsiteration},\ref{fig:rt2triepsiteration},\ref{fig:rt1quadepsiteration}.  This shows that, for fixed $k$, increasing $\epsilon$ also increases the iteration count.  On the other hand, for fixed $\epsilon$, one finds the largest iteration counts for intermediate values of the time step.  Much as the mesh-dependence study, we remark that varying the discretization order and cell shape has little effect on the results.

\begin{figure}[h]
        \begin{subfigure}[l]{0.475\textwidth}
                \begin{tikzpicture}[scale=0.88]
                \begin{semilogxaxis}[xlabel={$k$}, ylabel={Iterations},
                log basis x=10,
                ylabel near ticks,
                x tick label style={font=\tiny}]
                \%addlegendentry[text width=20pt,text depth=]
                \addplot[dotted,mark=otimes,mark options={solid, fill=gray}] table [x = k, y = iterations, col sep=comma]
                {Riesz.deg1.eps0.1.csv};
                \addplot[dashed,mark=*,mark options={solid,fill=gray}] table [x = k, y = iterations, col sep=comma]
                {Riesz.deg1.eps0.01.csv};
                \addplot[dashdotted,mark=o,mark options={solid}] table [x = k, y = iterations, col sep=comma]
                {Riesz.deg1.eps0.001.csv};
                \end{semilogxaxis}
                \end{tikzpicture}
                \caption{Using the bilinear form~\eqref{eq:goodip} (includes damping) as a preconditioner}
        \end{subfigure}
        \hspace{0.04\textwidth}
        \begin{subfigure}[r]{0.475\textwidth}
                \begin{tikzpicture}[scale=0.88]
                \begin{semilogxaxis}[xlabel={$k$},
                log basis x=10,
                ylabel near ticks,
                x tick label style={font=\tiny},
                legend cell align=left,
                legend style={overlay, at={(-0.16, 1.05)}, anchor=south},
                legend columns=5]
                \addlegendimage{empty legend}
                \addlegendentry[text width=20pt,text depth=]
                {$\epsilon =$};
                \addplot[dotted,mark=otimes,mark options={solid, fill=gray}] table [x = k, y = iterations, col sep=comma]
                {Riesz_Lite.deg1.eps0.1.csv};
                \addlegendentry{$0.1$}
                \addplot[dashed,mark=*,mark options={solid,fill=gray}] table [x = k, y = iterations, col sep=comma]
                {Riesz_Lite.deg1.eps0.01.csv};
                \addlegendentry{$0.01$}
                \addplot[dashdotted,mark=o,mark options={solid}] table [x = k, y = iterations, col sep=comma]
                {Riesz_Lite.deg1.eps0.001.csv};
                \addlegendentry{$0.001$}
                \end{semilogxaxis}
                \end{tikzpicture}
                \caption{Using the bilinear form~\eqref{eq:lessgoodip} (without damping) as a preconditioner}
        \end{subfigure}
  \caption{Iteration count with weighted-norm preconditioning as a function of $k$ and $\epsilon$ on a $128\times 128$ mesh divided into right triangles using lowest-order Raviart-Thomas elements.  Note that for a fixed $k$, the iteration count increases with decreasing $\epsilon$.  As in Figure~\ref{fig:rt1trivaryingmesh}, removing the damping term from the preconditioner leads to a very slight increase in iteration count.}
  \label{fig:rt1triepsiteration}
\end{figure}
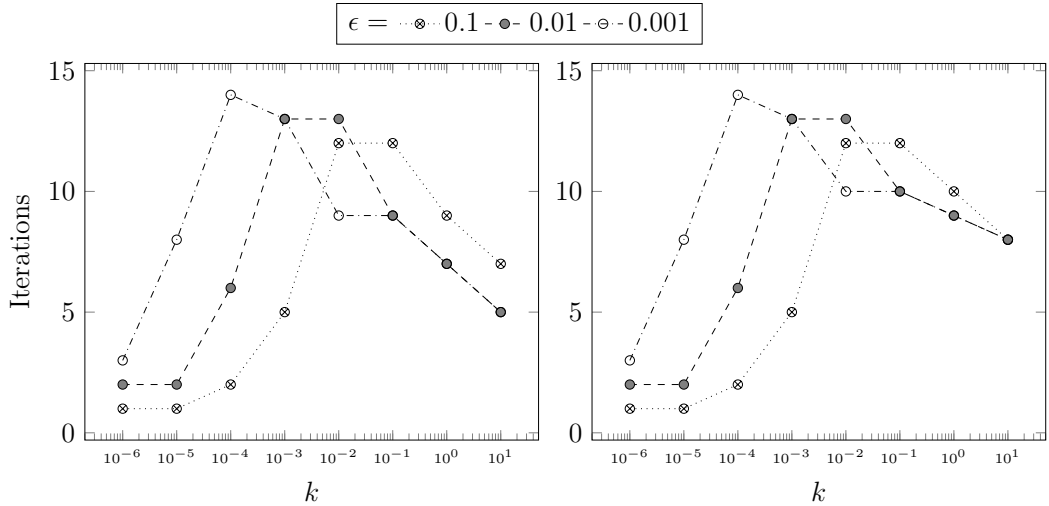

\begin{figure}[h]
        \begin{subfigure}[l]{0.475\textwidth}
                \begin{tikzpicture}[scale=0.88]
                \begin{semilogxaxis}[xlabel={$k$}, ylabel={Iterations},
                log basis x=10,
                ylabel near ticks,
                x tick label style={font=\tiny}]
                \%addlegendentry[text width=20pt,text depth=]
                \addplot[dotted,mark=otimes,mark options={solid, fill=gray}] table [x = k, y = iterations, col sep=comma]
                {Riesz.deg2.eps0.1.csv};
                \addplot[dashed,mark=*,mark options={solid,fill=gray}] table [x = k, y = iterations, col sep=comma]
                {Riesz.deg2.eps0.01.csv};
                \addplot[dashdotted,mark=o,mark options={solid}] table [x = k, y = iterations, col sep=comma]
                {Riesz.deg2.eps0.001.csv};
                \end{semilogxaxis}
                \end{tikzpicture}
                \caption{Using the bilinear form~\eqref{eq:goodip} (includes damping) as a preconditioner}
        \end{subfigure}
        \hspace{0.04\textwidth}
        \begin{subfigure}[r]{0.475\textwidth}
                \begin{tikzpicture}[scale=0.88]
                \begin{semilogxaxis}[xlabel={$k$},
                log basis x=10,
                ylabel near ticks,
                x tick label style={font=\tiny},
                legend cell align=left,
                legend style={overlay, at={(-0.16, 1.05)}, anchor=south},
                legend columns=5]
                \addlegendimage{empty legend}
                \addlegendentry[text width=20pt,text depth=]
                {$\epsilon =$};
                \addplot[dotted,mark=otimes,mark options={solid, fill=gray}] table [x = k, y = iterations, col sep=comma]
                {Riesz_Lite.deg2.eps0.1.csv};
                \addlegendentry{$0.1$}
                \addplot[dashed,mark=*,mark options={solid,fill=gray}] table [x = k, y = iterations, col sep=comma]
                {Riesz_Lite.deg2.eps0.01.csv};
                \addlegendentry{$0.01$}
                \addplot[dashdotted,mark=o,mark options={solid}] table [x = k, y = iterations, col sep=comma]
                {Riesz_Lite.deg2.eps0.001.csv};
                \addlegendentry{$0.001$}
                \end{semilogxaxis}
                \end{tikzpicture}
                \caption{Using the bilinear form~\eqref{eq:lessgoodip} (without damping) as a preconditioner}
        \end{subfigure}
  \caption{Repeating experiment in Figure~\ref{fig:rt1triepsiteration} with next-to-lowest order elements, showing little change in results.}
  \label{fig:rt2triepsiteration}
\end{figure}
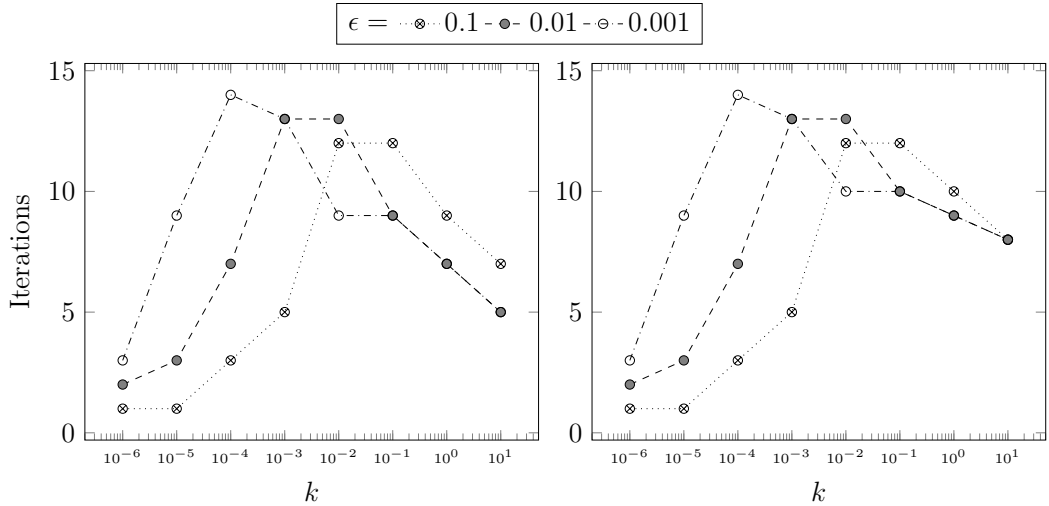

\vspace{1.0in}
\begin{figure}[h]
        \begin{subfigure}[l]{0.475\textwidth}
                \begin{tikzpicture}[scale=0.88]
                \begin{semilogxaxis}[xlabel={$k$}, ylabel={Iterations},
                log basis x=10,
                ylabel near ticks,
                x tick label style={font=\tiny}]
                \%addlegendentry[text width=20pt,text depth=]
                \addplot[dotted,mark=otimes,mark options={solid, fill=gray}] table [x = k, y = iterations, col sep=comma]
                {Riesz.deg1.eps0.1.csv};
                \addplot[dashed,mark=*,mark options={solid,fill=gray}] table [x = k, y = iterations, col sep=comma]
                {Riesz.deg1.eps0.01.csv};
                \addplot[dashdotted,mark=o,mark options={solid}] table [x = k, y = iterations, col sep=comma]
                {Riesz.deg1.eps0.001.csv};
                \end{semilogxaxis}
                \end{tikzpicture}
                \caption{Using the bilinear form~\eqref{eq:goodip} (includes damping) as a preconditioner}
        \end{subfigure}
        \hspace{0.04\textwidth}
        \begin{subfigure}[r]{0.475\textwidth}
                \begin{tikzpicture}[scale=0.88]
                \begin{semilogxaxis}[xlabel={$k$},
                log basis x=10,
                ylabel near ticks,
                x tick label style={font=\tiny},
                legend cell align=left,
                legend style={overlay, at={(-0.16, 1.05)}, anchor=south},
                legend columns=5]
                \addlegendimage{empty legend}
                \addlegendentry[text width=20pt,text depth=]
                {$\epsilon =$};
                \addplot[dotted,mark=otimes,mark options={solid, fill=gray}] table [x = k, y = iterations, col sep=comma]
                {Riesz_Lite.deg1.eps0.1.csv};
                \addlegendentry{$0.1$}
                \addplot[dashed,mark=*,mark options={solid,fill=gray}] table [x = k, y = iterations, col sep=comma]
                {Riesz_Lite.deg1.eps0.01.csv};
                \addlegendentry{$0.01$}
                \addplot[dashdotted,mark=o,mark options={solid}] table [x = k, y = iterations, col sep=comma]
                {Riesz_Lite.deg1.eps0.001.csv};
                \addlegendentry{$0.001$}
                \end{semilogxaxis}
                \end{tikzpicture}
                \caption{Using the bilinear form~\eqref{eq:lessgoodip} (without damping) as a preconditioner}
        \end{subfigure}
        \caption{Iteration count with weighted-norm preconditioning as a function of $k$ and $\epsilon$ on a $128\times 128$ mesh of squares using lowest-order Raviart-Thomas elements.  Again, results are nearly identical to the two triangular cases.}
        \label{fig:rt1quadepsiteration}
\end{figure}
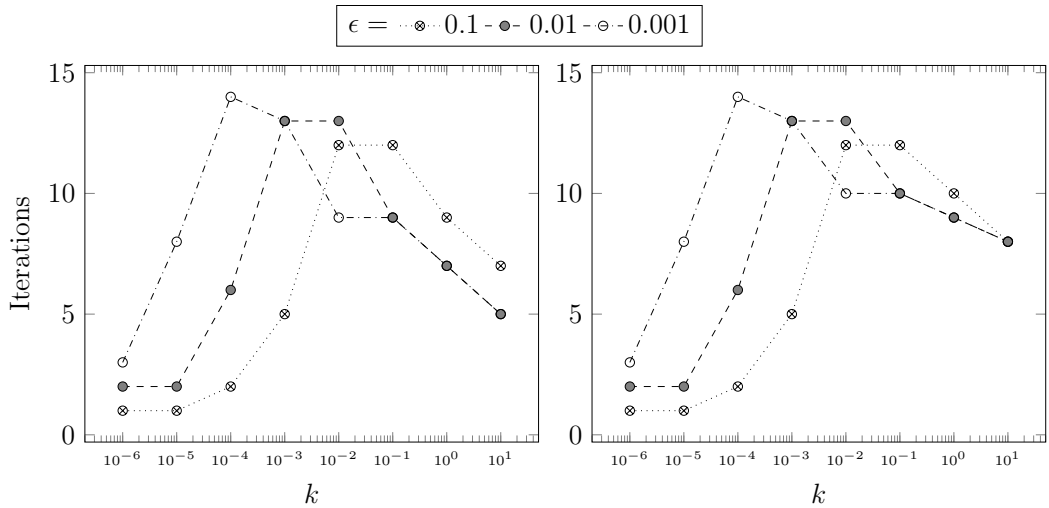

Now, we consider inexact application of the inverse of the top left block of our preconditioner by some kind of multi-level method  Instead of typical pointwise smoothers for problems in $H^1$, the geometric multigrid of Arnold, Falk, and Winther in~\cite{arnold2000multigrid} requires one to solve local problems on cells surrounding vertex patches.  This approach is accessible in Firedrake through the high-level solver interface described in~\cite{kirby2018solver} and the \texttt{pcpatch} package developed under PETSc~\cite{farrell2019pcpatch}.
Kolev and Vassilevski~\cite{kolev2006parallel} also present algebraic multigrid approach.  Originally for $\hcurl$, their method readily adapts to two-dimensional $\hdiv$ problems.  This is an algebraic auxiliary space method based on~\cite{hiptmair2007nodal} that requires the user to set a discrete gradient operator and mesh vertex coordinates and internally solves a (possibly algebraically derived) Poisson-type equation.  The hypre implementation is available through PETSc, and we have also developed a Python wrapper that extracts the mesh vertices and computes a discrete gradient for each mesh and configures the underlying PETSc preconditioner appropriately.  However, we report only results using the geometric multigrid variant here.

In our results, instead of applying the inverse of $P_{V_h}$ via LU factorization at each outer iteration, we apply a single sweep of full multigrid using four levels of refinement.  On each level, we apply one step of Richardson smoothing using the vertex-patch preconditioner described in~\cite{arnold2000multigrid, farrell2019pcpatch}.  A sparse direct method is used on the coarsest mesh.  The PETSc paramters are the same as used for the $\hdiv$ Riesz map in~\cite{farrell2019pcpatch}.   Because our sample problems are not large enough for the asymptotic complexity of multigrid to beat the sparse direct solve, we continue to only report iteration counts rather than timings.

\vspace{1.0in}
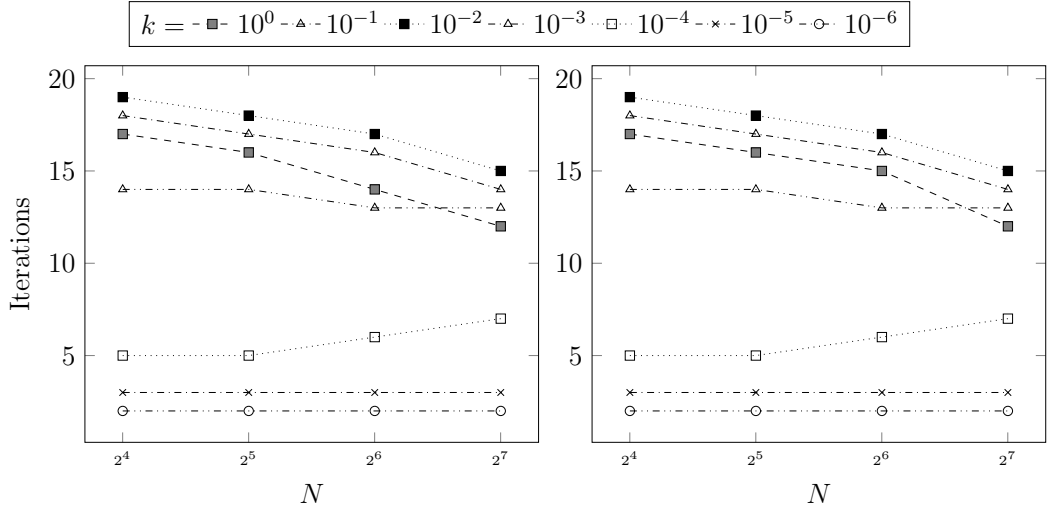
\begin{figure}[h]
        \begin{subfigure}[l]{0.475\textwidth}
                \begin{tikzpicture}[scale=0.88]
                \begin{semilogxaxis}[xlabel={$N$}, ylabel={Iterations},
                log basis x=2,
                ylabel near ticks,
                x tick label style={font=\tiny}]
                \addplot[dashed,mark=square*,mark options={solid,fill=gray}] table [x = N, y = iterations, col sep=comma]
                {tri.mg.Riesz.deg1.k1.0.csv};
                \addplot[dashdotted,mark=triangle,mark options={solid}] table [x = N, y = iterations, col sep=comma]
                {tri.mg.Riesz.deg1.k0.1.csv};
                \addplot[dotted,mark=square*, mark options={solid,fill}] table [x = N, y = iterations, col sep=comma]
                {tri.mg.Riesz.deg1.k0.01.csv};
                \addplot[dash dot dot,mark=triangle, mark options={solid,fill}] table [x = N, y = iterations, col sep=comma]
                {tri.mg.Riesz.deg1.k0.001.csv};
                \addplot[dotted,mark=square, mark options={solid}] table [x = N, y = iterations, col sep=comma]
                {tri.mg.Riesz.deg1.k0.0001.csv};
                \addplot[dashdotted,mark=x, mark options={solid}] table [x = N, y = iterations, col sep=comma]
                {tri.mg.Riesz.deg1.k1e-05.csv};
                \addplot[dash dot dot,mark=o, mark options={solid,fill}] table [x = N, y = iterations, col sep=comma]
                {tri.mg.Riesz.deg1.k1e-06.csv};
                \end{semilogxaxis}
                \end{tikzpicture}
                \caption{Using the bilinear form~\eqref{eq:goodip} (includes damping) as a preconditioner}
       \end{subfigure}
        \hspace{0.04\textwidth}
        \begin{subfigure}[r]{0.475\textwidth}
                \begin{tikzpicture}[scale=0.88]
                \begin{semilogxaxis}[xlabel={$N$},
                log basis x=2,
                ylabel near ticks,
                x tick label style={font=\tiny},
                legend cell align=left,
                legend style={overlay, at={(-0.16, 1.05)}, anchor=south},
                legend columns=9]
                \addlegendimage{empty legend}
                \addlegendentry[text width=20pt,text depth=]
                {$k =$};
                \addplot[dashed,mark=square*,mark options={solid,fill=gray}] table [x = N, y = iterations, col sep=comma]
                {tri.mg.Riesz_Lite.deg1.k1.0.csv};
                \addlegendentry{$10^0$}
                \addplot[dashdotted,mark=triangle,mark options={solid}] table [x = N, y = iterations, col sep=comma]
                {tri.mg.Riesz_Lite.deg1.k0.1.csv};
                \addlegendentry{$10^{-1}$}
                \addplot[dotted,mark=square*, mark options={solid,fill}] table [x = N, y = iterations, col sep=comma]
                {tri.mg.Riesz_Lite.deg1.k0.01.csv};
                \addlegendentry{$10^{-2}$}
                \addplot[dash dot dot,mark=triangle, mark options={solid,fill}] table [x = N, y = iterations, col sep=comma]
                {tri.mg.Riesz_Lite.deg1.k0.001.csv};
                \addlegendentry{$10^{-3}$}
                \addplot[dotted,mark=square, mark options={solid}] table [x = N, y = iterations, col sep=comma]
                {tri.mg.Riesz_Lite.deg1.k0.0001.csv};
                \addlegendentry{$10^{-4}$}
                \addplot[dashdotted,mark=x, mark options={solid}] table [x = N, y = iterations, col sep=comma]
                {tri.mg.Riesz_Lite.deg1.k1e-05.csv};
                \addlegendentry{$10^{-5}$}
                \addplot[dash dot dot,mark=o, mark options={solid,fill}] table [x = N, y = iterations, col sep=comma]
               {tri.mg.Riesz_Lite.deg1.k1e-06.csv};
                \addlegendentry{$10^{-6}$}
                \end{semilogxaxis}
                \end{tikzpicture}
                \caption{Using the bilinear form~\eqref{eq:lessgoodip} (without damping) as a preconditioner}
        \end{subfigure}
        \caption{Iteration count versus mesh refinement under various $k$ values for $C=f=1$, $\beta=0.1$, and $\epsilon=0.01$ using lowest-order triangular Raviart-Thomas elements.  Instead of inverting $P_{V_h}$ by LU factorization, however, a single full multigrid cycle is used.  Comparing to  Figure~\ref{fig:rt1trivaryingmesh} reveals a slight increase in iteration count in exchange for forgoing the sparse direct factorization.  }  
        \label{fig:rt1trivaryingmesh_mg}
\end{figure}

\vspace{1.0in}
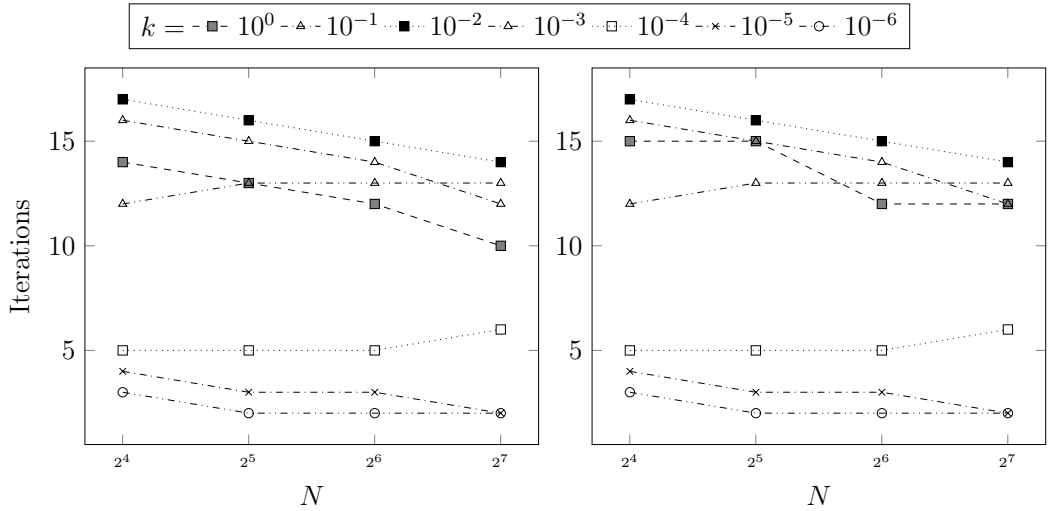
\begin{figure}[h]
        \begin{subfigure}[l]{0.475\textwidth}
                \begin{tikzpicture}[scale=0.88]
                \begin{semilogxaxis}[xlabel={$N$}, ylabel={Iterations},
                log basis x=2,
                ylabel near ticks,
                x tick label style={font=\tiny}]
                \addplot[dashed,mark=square*,mark options={solid,fill=gray}] table [x = N, y = iterations, col sep=comma]
                {quad.mg.Riesz.deg1.k1.0.csv};
                \addplot[dashdotted,mark=triangle,mark options={solid}] table [x = N, y = iterations, col sep=comma]
                {quad.mg.Riesz.deg1.k0.1.csv};
                \addplot[dotted,mark=square*, mark options={solid,fill}] table [x = N, y = iterations, col sep=comma]
                {quad.mg.Riesz.deg1.k0.01.csv};
                \addplot[dash dot dot,mark=triangle, mark options={solid,fill}] table [x = N, y = iterations, col sep=comma]
                {quad.mg.Riesz.deg1.k0.001.csv};
                \addplot[dotted,mark=square, mark options={solid}] table [x = N, y = iterations, col sep=comma]
                {quad.mg.Riesz.deg1.k0.0001.csv};
                \addplot[dashdotted,mark=x, mark options={solid}] table [x = N, y = iterations, col sep=comma]
                {quad.mg.Riesz.deg1.k1e-05.csv};
                \addplot[dash dot dot,mark=o, mark options={solid,fill}] table [x = N, y = iterations, col sep=comma]
                {quad.mg.Riesz.deg1.k1e-06.csv};
                \end{semilogxaxis}
                \end{tikzpicture}
                \caption{Using the bilinear form~\eqref{eq:goodip} (includes damping) as a preconditioner}
       \end{subfigure}
        \hspace{0.04\textwidth}
        \begin{subfigure}[r]{0.475\textwidth}
                \begin{tikzpicture}[scale=0.88]
                \begin{semilogxaxis}[xlabel={$N$},
                log basis x=2,
                ylabel near ticks,
                x tick label style={font=\tiny},
                legend cell align=left,
                legend style={overlay, at={(-0.16, 1.05)}, anchor=south},
                legend columns=9]
                \addlegendimage{empty legend}
                \addlegendentry[text width=20pt,text depth=]
                {$k =$};
                \addplot[dashed,mark=square*,mark options={solid,fill=gray}] table [x = N, y = iterations, col sep=comma]
                {quad.mg.Riesz_Lite.deg1.k1.0.csv};
                \addlegendentry{$10^0$}
                \addplot[dashdotted,mark=triangle,mark options={solid}] table [x = N, y = iterations, col sep=comma]
                {quad.mg.Riesz_Lite.deg1.k0.1.csv};
                \addlegendentry{$10^{-1}$}
                \addplot[dotted,mark=square*, mark options={solid,fill}] table [x = N, y = iterations, col sep=comma]
                {quad.mg.Riesz_Lite.deg1.k0.01.csv};
                \addlegendentry{$10^{-2}$}
                \addplot[dash dot dot,mark=triangle, mark options={solid,fill}] table [x = N, y = iterations, col sep=comma]
                {quad.mg.Riesz_Lite.deg1.k0.001.csv};
                \addlegendentry{$10^{-3}$}
                \addplot[dotted,mark=square, mark options={solid}] table [x = N, y = iterations, col sep=comma]
                {quad.mg.Riesz_Lite.deg1.k0.0001.csv};
                \addlegendentry{$10^{-4}$}
                \addplot[dashdotted,mark=x, mark options={solid}] table [x = N, y = iterations, col sep=comma]
                {quad.mg.Riesz_Lite.deg1.k1e-05.csv};
                \addlegendentry{$10^{-5}$}
                \addplot[dash dot dot,mark=o, mark options={solid,fill}] table [x = N, y = iterations, col sep=comma]
               {quad.mg.Riesz_Lite.deg1.k1e-06.csv};
                \addlegendentry{$10^{-6}$}
                \end{semilogxaxis}
                \end{tikzpicture}
                \caption{Using the bilinear form~\eqref{eq:lessgoodip} (without damping) as a preconditioner}
        \end{subfigure}
        \caption{Iteration count versus mesh refinement under various $k$ values for $C=f=1$, $\beta=0.1$, and $\epsilon=0.01$ using lowest-order Raviart-Thomas elements on squares.  A full multigrid cycle using four levels, as in Figure~\ref{fig:rt1trivaryingmesh_mg}, is used instead of sparse direct factorization of $P_{V_h}$, again resulting in a slight increase in iteration count.}  
        \label{fig:rt1quadvaryingmesh_mg}
\end{figure}

As a final example, we consider a problem with nonlinear damping.  In particular, we choose $g(u) = |u|^2 u$ in~\eqref{eq:nonlinweakform} (this bypasses numerical wrinkles in differentiating through the singularity of quadratic damping).  Our typical use case is in time-stepping, where the solution at the previous time step serves as an initial guess for Newton iteration.  To imitate having a close initial guess, we seed Newton's method with the solution of the linear, undamped problem.  In each case, we observed that Newton requires but a single iteration to converge, suggesting that there is not a significant need to split the nonlinear term from the rest of the equation for implicit time-stepping to be effective.

\vspace{1.0in}
\begin{figure}[h]
        \begin{subfigure}[l]{0.475\textwidth}
                \begin{tikzpicture}[scale=0.88]
                \begin{semilogxaxis}[xlabel={$N$}, ylabel={Iterations},
                log basis x=2,
                ylabel near ticks,
                x tick label style={font=\tiny}]
                  \addplot[dashed,mark=square*,mark options={solid,fill=gray}] table [x = N, y = Lin, col sep=comma]
                {tri.nonlin.Riesz.deg1.k1.0.csv};
                \addplot[dashdotted,mark=triangle,mark options={solid}] table [x = N, y = Lin, col sep=comma]
                {tri.nonlin.Riesz.deg1.k0.1.csv};
                \addplot[dotted,mark=square*, mark options={solid,fill}] table [x = N, y = Lin, col sep=comma]
                {tri.nonlin.Riesz.deg1.k0.01.csv};
                \addplot[dash dot dot,mark=triangle, mark options={solid,fill}] table [x = N, y = Lin, col sep=comma]
                {tri.nonlin.Riesz.deg1.k0.001.csv};
                \addplot[dotted,mark=square, mark options={solid}] table [x = N, y = Lin, col sep=comma]
                {tri.nonlin.Riesz.deg1.k0.0001.csv};
                \addplot[dashdotted,mark=x, mark options={solid}] table [x = N, y = Lin, col sep=comma]
                {tri.nonlin.Riesz.deg1.k1e-05.csv};
                \addplot[dash dot dot,mark=o, mark options={solid,fill}] table [x = N, y = Lin, col sep=comma]
                {tri.nonlin.Riesz.deg1.k1e-06.csv};
                \end{semilogxaxis}
                \end{tikzpicture}
                \caption{Using the bilinear form~\eqref{eq:goodip} (includes damping) as a preconditioner}
       \end{subfigure}
        \hspace{0.04\textwidth}
        \begin{subfigure}[r]{0.475\textwidth}
                \begin{tikzpicture}[scale=0.88]
                \begin{semilogxaxis}[xlabel={$N$},
                log basis x=2,
                ylabel near ticks,
                x tick label style={font=\tiny},
                legend cell align=left,
                legend style={overlay, at={(-0.16, 1.05)}, anchor=south},
                legend columns=9]
                \addlegendimage{empty legend}
                \addlegendentry[text width=20pt,text depth=]
                               {$k =$};
                \addplot[dashed,mark=square*,mark options={solid,fill=gray}] table [x = N, y = Lin, col sep=comma]
                {tri.nonlin.Riesz_Lite.deg1.k1.0.csv};
                \addlegendentry{$10^0$}
                \addplot[dashdotted,mark=triangle,mark options={solid}] table [x = N, y = Lin, col sep=comma]
                {tri.nonlin.Riesz_Lite.deg1.k0.1.csv};
                \addlegendentry{$10^{-1}$}
                \addplot[dotted,mark=square*, mark options={solid,fill}] table [x = N, y = Lin, col sep=comma]
                {tri.nonlin.Riesz_Lite.deg1.k0.01.csv};
                \addlegendentry{$10^{-2}$}
                \addplot[dash dot dot,mark=triangle, mark options={solid,fill}] table [x = N, y = Lin, col sep=comma]
                {tri.nonlin.Riesz_Lite.deg1.k0.001.csv};
                \addlegendentry{$10^{-3}$}
                \addplot[dotted,mark=square, mark options={solid}] table [x = N, y = Lin, col sep=comma]
                {tri.nonlin.Riesz_Lite.deg1.k0.0001.csv};
                \addlegendentry{$10^{-4}$}
                \addplot[dashdotted,mark=x, mark options={solid}] table [x = N, y = Lin, col sep=comma]
                {tri.nonlin.Riesz_Lite.deg1.k1e-05.csv};
                \addlegendentry{$10^{-5}$}
                \addplot[dash dot dot,mark=o, mark options={solid,fill}] table [x = N, y = Lin, col sep=comma]
               {tri.nonlin.Riesz_Lite.deg1.k1e-06.csv};
                \addlegendentry{$10^{-6}$}
                \end{semilogxaxis}
                \end{tikzpicture}
                \caption{Using the bilinear form~\eqref{eq:lessgoodip} (without damping) as a preconditioner}
        \end{subfigure}
        \caption{Iteration count versus mesh refinement under various $k$ values for $f=1$, $\beta=0.1$, and $\epsilon=0.01$ for the nonlinear damping law $g(u) = |u|^2 u$.}  
        \label{fig:rt1trivaryingmesh_nonlin}
\end{figure}
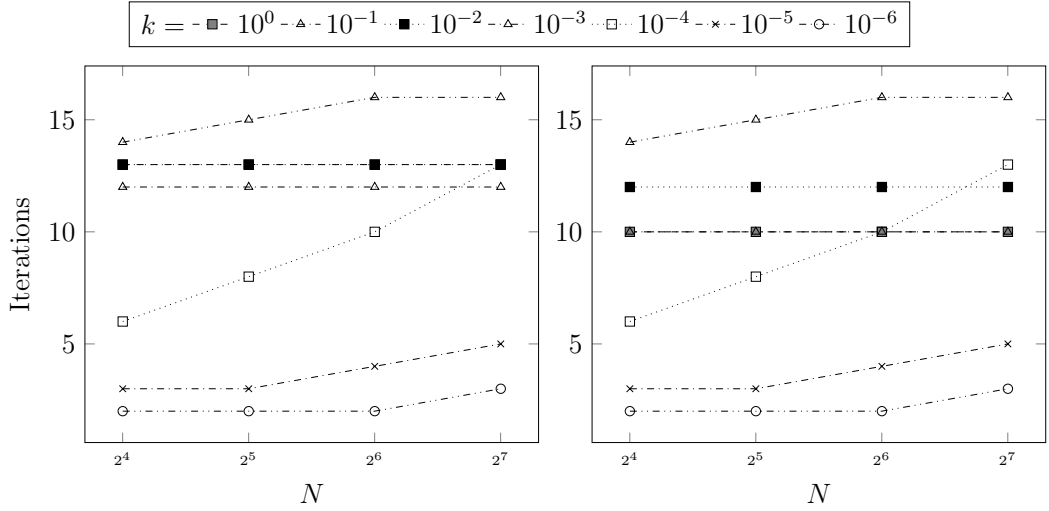

\section{Conclusions}
\label{sec:conc}
We have developed effective weighted-norm preconditioners for a mixed finite element/Crank-Nicolson discretization of the linearized rotating shallow water equations with (possibly nonlinear) damping.  These preconditioners are based on defining a suitable inner product in which the operators are bounded with bounded inverse in a relatively paramater-independent way.  These estimates in turn control the spectrum of the preconditioned operator.  Our estimates remain dependent on the ratio $\tfrac{k}{\epsilon}$, although this seems relatively benign in practice.  Moreover, inexactly applying the preconditioner through a multigrid sweep and neglecting damping terms in the inner product lead to further simplifications with only mild effects on iteration count.

This work suggests many future research directions.  Since our theory
and numerical observations both seem independent of mesh type and
discretization order, we hope to apply these preconditioners to
unstructured quadrilateral elements such as Arbogast-Correa~\cite{arbogast2016two}.  Moreover,
our techniques should be applicable to more complex tide models that
might include additional nonlinearities or layering.

\bibliographystyle{elsarticle-num} 
\bibliography{bib}





\end{document}